\documentclass{amsart} 

\usepackage{amsmath,amssymb,mathrsfs,amsthm,delarray}
\usepackage{pictex}
\usepackage[all,cmtip]{xy}
\numberwithin{equation}{section}

\newtheorem{Theorem}{Theorem}[section]

\newtheorem*{theorem3'}{Theorem 2'}
\newtheorem*{Theorem A}{Theorem A}
\newtheorem*{Theorem A'}{Theorem A'}
\newtheorem*{Theorem B}{Theorem B}
\newtheorem*{Theorem C}{Theorem C}
\newtheorem{Proposition}{Proposition}[section]
\newtheorem{Lemma}{Lemma}[section]
\newtheorem{definition}{Definition}[section]
\newtheorem{Coro}{Corollary}[section]

\newtheorem{Remark}{Remark}[section]

\newcommand{\C}{\mathbb C}

\newcommand{\R}{\mathbb{R}}
\newcommand{\Z}{\mathbb{Z}}

\usepackage{geometry} 
\usepackage{graphicx} 

\title{Primitive tuning for  non-hyperbolic polynomials}
\author{Yimin Wang}
\date{\today}
\address{Shanghai Center for Mathematical Sciences,
Fudan University, No. 2005 Songhu Road, Shanghai 200438, China}
\email{yiminwang16@fudan.edu.cn}
\begin{document}
\maketitle

\begin{abstract}  Let $f_0$ be a polynomial of degree $d_1+d_2$ $(d_1\ge 2, d_2\ge 1)$with a periodic critical point $0$ of multiplicity $d_1-1$ and a Julia critical point of multiplicity $d_2$. We show that if $f_0$ is primitive, free of neutral periodic points and non-renormalizable at the Julia critical point,   then the straightening map $\chi_{f_0}:\mathcal C(\lambda_{f_0}) \to \mathcal C_{d_1}$ is a bijection. More precisely,  $f^{m_0}$ has a polynomial-like restriction which is hybrid equivalent to some polynomial in $\mathcal C_{d_1}$ for each map $f \in \mathcal C(\lambda_{f_0})$, where $m_0$ is the period of $0$ under $f_0$. On the other hand, $f_0$ can be tuned with any polynomial $g\in \mathcal C_{d_1}$.  As a consequence, we conclude that the straightening map  $\chi_{f_0}$ is a homeomorphism from $\mathcal C(\lambda_{f_0})$ onto the Mandelbrot set when $d_1=2$.
This together with the main result in~\cite{SW} solve the problem for primitive tuning for cubic polynomials with connected Julia sets thoroughly.
\end{abstract}
\section{Introduction}
 Tuning and straightening  were first introduced by Douady-Hubbard~\cite{DH2} to prove the existence of baby Mandelbrot sets in the quadratic family. Roughly speaking, tuning is a procedure to replace the bounded superattracting Fatou components with the copies of a filled Julia set of another polynomial and respect some combinatorial properties.
Douady-Hubbard proved any quadratic polynomial which has a periodic critical point can be tuned with any quadratic polynomials in the Mandelbrot set $\mathcal M$. 

In~\cite{IK}, Inou-Kiwi generalized the definition of the straightening maps for higher degree polynomials. They used some combinatorial argument to show that any primitive postcritically finite hyperbolic polynomial can be tuned with any hyperbolic generalized polynomial with fiber-wise connected Julia set. In~\cite{SW}, Shen and the author used quasiconformal surgery and ran Thurston's Algorithm to prove that any primitive postcritically finite hyperbolic polynomial can be tuned with any  generalized polynomial with fiber-wise connected Julia set. Therefore, the problem for primitive tuning had been solved thoroughly for hyperbolic  polynomials.

The question which remains is whether primitive non-hyperbolic  polynomials can be tuned with  generalized polynomials or not. Let us consider the family $\mathcal P_d$ of  polynomials of degree $d$ $(d\ge 2)$ normalized as the following form
\[f:z\mapsto z^{d}+a^{d-1}z^{d-1}+\cdots +a_2z^2+a_0.\]
Note that each map $f\in \mathcal P_d$ is monic and has a critical point $0$.
The set $$\mathcal C:=\{f\in \mathcal P_d\mid J_f~\text{is connected}\}$$
 is called the {\em connectedness locus of degree $d$}. 
Let $f_0\in \mathcal C_{d_1+d_2}$ be a polynomial with a periodic critical point $0$ of multiplicity $d_1-1$ and a Julia critical point $c_0$ of multiplicity $d_2$. Let $\lambda_{f_0}$ denote the rational lamination of $f_0$. We believe that if $f_0$ is renormalizable at the free critical point $c_0$, then there is a hyperbolic post-critically finite polynomial $f_*\in \mathcal C_{d_1+d_2}$ with non-trivial rational lamination such that $\lambda_{f_*}\subset \lambda_{f_0}$. 
The primitive tuning for such a hyperbolic post-critically finite polynomial $f_*$ had been studied thoroughly in \cite{SW}. So we will restrict our attention to the case that $f_0$ is non-renormalizable at the free critical point $c_0$. There is no loss of generality in assuming $f_0$ is free of neutral periodic points by \cite[Theorem~5.17]{IK}. Let  $\mathcal C(\lambda_{f_0})$ denote the set of {\em combinatorially $\lambda_{f_0}$- renormalizable} maps. For the precise definition of $\mathcal C(\lambda_{f_0})$, see \S\ref{sec:tuning}. Set $g\in \mathcal C_{d_1}$. We say $f_0$ {\em can be tuned with $g$ } if there exists $f\in \mathcal C(\lambda_{f_0})$ such that $f^{m_0}$ has a polynomial-like restriction which is hybrid equivalent to $g$ (and the hybrid conjugacy respects the corresponding external markings). Such an $f$ is called a {\em tuning} of $f_0$ by $g$. We say $f_0$ is {\em primitive} if any two Fatou components of $f_0$ have disjoint closures.

The aim of this paper is to show:
\begin{Theorem A} Let $f_0\in \mathcal C_{d_1+d_2}$ be a primitive polynomial with a periodic critical point $0$ of multiplicity $d_1-1$ and a Julia critical point $c_0$ of  multiplicity $d_2$. Fix an internal angle system $\alpha$ for $f_0$. If $f_0$ has no neutral periodic points  and $f_0$ is non-renormalizable at $c_0$, then the straightening map $\chi_{f_0}:\mathcal C(\lambda_{f_0}) \to \mathcal C_{d_1}$  induced by $f_{0}$ and $\alpha$  is a bijection.
\end{Theorem A}

The proof falls naturally into two parts: surjectivity and injectivity. To show the surjectivity of $\chi_{f_0}$, it suffices to show:

\begin{Theorem B} Under the assumptions in the Theorem~A, $f_0$ can be tuned with any polynomial  $g\in\mathcal C_{d_1}$.
\end{Theorem B}
The basic idea is to use qc surgery and apply \cite[Theorem~5.1]{SW}. The main difficulty is that we cannot construct a desired tuning by using qc surgery and Thurston's Algorithm directly. In~\cite{SW}, one can construct a quasiregular map with certain dynamical properties so that all the ramification points lie in the union of all the small filled Julia sets. Thus the postcritical set will not affect the combinatorics (the landing relation for external rays). However, $f_0$ has a free Julia critical point in our case. This does affect the combinatorics when we do qc surgery.  To overcome this difficulty, we construct the tuning by approximation instead of constructing it directly. More precisely, we first control the orbit of the free critical point carefully when we do qc surgery to show there is a sequence $\{f_k\}$ of polynomials such that $f^m_k$ is hybrid equivalent to a given $g\in \mathcal C_{d_1}$ and $f_k$ has the same combinatorics as $f_0$ up to depth $k$ for all $k\in \mathbb N$. Then we show the limit $f$ of $\{f_k\}$ is our desired tuning.

It is worth pointing out that Douady-Hubbard's method does not work since $\mathcal C_{d_1}$ may  contain multi-critical maps. Even for the case $d_1=2$, one cannot apply Douady-Hubbard's Theorem directly since $\mathcal C(\lambda_{f_0})$ is not an analytic family in general.

The proof of the injectivity of $\chi_{f_0}$ is based on the arguments developed in~\cite{AKLS}. For any $f,\widetilde f\in \mathcal C(\lambda_{f_0})$ with $\chi_{f_0}(f)=\chi_{f_0}(\widetilde f)$, we show that $f$ and $\widetilde f$ are conformally conjugate. To this end, We first construct a sequence $\{\Phi_k\}$ of qc maps such that $\Phi_k$  is a weak pseudo-conjugacy between $f$ and $\widetilde f$ up to  depth $k$ for all $k\in \mathbb N$. Then we prove some technical lemmas to show that there exists a sequence $\{k_n\}$ so that we can make $\Phi_{k_n}$ be a pseudo-conjugacy between $f$ and $\widetilde f$ up to depth $k_n$ by adjustment. Moreover, $\{\Phi_{k_n}\}$ can be made uniformly qc. Finally, we show the limit $\Phi$ of $\{\Phi_{k_n}\}$ is a conformal conjugacy between $f$ and $\widetilde f$. 

Combining Theorem~A,~\cite[proposition~4.7]{McM2} and~\cite[Theorem~8.1]{IK}, we conclude:
\begin{Theorem C}Under the assumptions in Theorem~A, if $d_1=2$ then the straightening map $\chi_{f_0}:\mathcal C(\lambda_{f_0}) \to \mathcal C_2=\mathcal M_2$ is a homeomorphism, where $\mathcal M_2$ is the Mandelbrot set.
\end{Theorem C}
 
Theorem~C together with the main results in \cite{SW} solve the problem for primitive tuning for cubic maps with connected Julia sets thoroughly. For more related topics on qc surgery, we refer the readers to \cite{BF,EY,Shishi}.

The paper is organized as follows. In \S\ref{sec:tuning}, we review some known facts on rational laminations and give the precise definition of tuning. In \S\ref{sec:puzzle} and \S\ref{sec:Kahn}, we follow the ideas in~\cite{SW} to construct a specific Yoccoz puzzle which satisfies some certain dynamical properties and prove a useful lemma on qc distortion. Our case is slightly different from that of~\cite{SW} by virtue of the existence of  the free Julia critical point. But it is fair to say the proofs are similar in spirit. The readers who are familiar with the results in~\cite{SW} may skip the proofs in these two section. In \S~5, we build up qc weak pseudo-conjugacies between $f$ and maps in $\mathcal C(\lambda_{f_0})$, which turns out to be one of the most powerful tool in the proof of Theorem~A.  The aim of section~\ref{sec:qc surgery} is to prove Theorem~B.  To this end, we first use qc surgery and apply \cite[Theorem~5.1]{SW} to construct  a sequence $\{f_k\}$ from $f_0$ and a given polynomial  $g\in \mathcal C_{d_1}$ such that $f_k$ has the same combinatorics as $f_0$ up to depth $k$ and $f^{m_0}_k$ has a polynomial-like restriction hybrid equivalent to $g$ for all $k\in \mathbb N$. Then we show the limit of such a sequence is a tuning of $f_0$ by $g$. In \S\ref{sec:homeomorphism}, we prove some technical lemmas and then prove Theorem~A.

\medskip
\noindent
{\bf Acknowledgment.} The author would like to thank Hiroyuki Inou and Weixiao Shen for helpful discussions and comments. This work was supported by {\bf CSC}.

\section{Laminations and Tuning}\label{sec:tuning}
Throughout this paper we will fix a primitive non-hyperbolic  polynomial $f_0 \in \mathcal C_{d_1+d_2}$ with a periodic critical point $0$ of multiplicity $d_1-1$ and a free Julia critical point $c_0$ of multiplicity $d_2$. \par
In this section, we will recall some background about laminations and give the definition of tuning.

\subsection{Rational laminations}For each polynomial $f\in \mathcal P_{d}$, the {\em Green function} is defined as
$$G_f(z)=\lim_{n\to\infty} \frac{1}{{d}^n}\log^+ |f^n(z)|,$$
where $\log^+=\max(\log , 0).$

The set 
\[D_f(\infty):=\{z\mid f^n(z)\to \infty\}\]
is called the {\em basin of infinity} of $f$.
The complement  $K(f)=K_f=\C\setminus D_f(\infty)$ of the basin of infinity is called the {\em filled Julia set} of $f$. The common boundary of $K(f)$ and $D_f(\infty)$ is called  the {\em Julia set} of $f$ and denoted by $J(f)=J_f$.

There exists a unique conformal map \[\phi_f:D_f(\infty)\setminus \{z\in \C\mid G_f(z)\le r_f\} \to \{z\mid |z|>r_f\}\]
such that $\phi(z)/z\to 1$ as $z\to\infty$ and such that $\phi_f\circ f(z)=(\phi_f(z))^d$ on $D_f(\infty)\setminus \{z\in \C\mid G_f(z)\le r_f\}$, where $r_f=\max \{G_{f}(c)\mid c~\text{is a critical point of}~ f\}$. Moreover, $G_f(z)=\log|\phi_f(z)|$.

The {\em equipotential curve of height $h$ of $f$} is defined as 
\[E_{f}(h)=\{z\mid G_f(z)=h\},\]
and the  {\em external ray of angle $t \in \R/\Z$}  as
\[\mathcal R_{f}(t)=\{\phi_f^{-1}(re^{i2\pi t})\mid r_f<r<\infty\}.\]
An external ray $R_f(t)$ which is not bifurcated has a smooth extension: $\mathcal R_{f}(t)=\{\phi_f^{-1}(re^{i2\pi t})\mid 1<r<\infty\}$.
We say the external ray $\mathcal R_f(t)$ {\em lands} at some point $z_0$ if it is not bifurcated and $\lim_{r\to 1}\phi_f^{-1}(re^{i2\pi t})=z_0$. 
\begin{definition}[Rational lamination]
 The {rational lamination} $\lambda_f$ of $f \in \mathcal C_d$ is the equivalence relation on $\mathbb Q/\Z$ so that $s\sim t$ if and only if $\mathcal R_f(s)$ and $\mathcal R_f(t)$  land at a common point. 
 \end{definition}

We say $f\in \mathcal C_d$ is {\em combinatorially $\lambda_{f_0}$-renormalizable} if $\lambda_f\supset \lambda_{f_0}$. Let $\mathcal C(\lambda_{f_0})$ denote the set of all the  combinatorially $\lambda_{f_0}$-renormalizable maps. By~ \cite[Theorem~8.1]{IK}, $\mathcal C(\lambda_{f_0})$ is compact.

\subsection{Renormalization and Tuning}Recall that {\em a polynomial-like map} is a holomorphic proper map $g:U\to V $, where $U\Subset V$  are quasidisks  in $\C$. We call the set $K(g):=\{z\in U\mid f^k(z)\in U~\text{for all}~k\in\mathbb N\}$ {\em the filled Julia set of $g$}.  Two polynomial-like maps $g$ and $\widetilde g$ are said to be {\em qc conjugate} if there exists a qc map $\varphi:\C\to \C$ such that $\varphi\circ g=\widetilde g\circ \varphi$ near $K(g)$. We say $g$ and $\widetilde g$ are {\em hybrid equivalent} if they are qc conjugate and there exists a corresponding qc conjugacy $\varphi$ between them such that $\partial\bar \varphi=0$ a.e. on $K(g)$. For a more detailed treatment,  we refer the reader to~\cite{DH2,L2}.

Let $f\in \mathcal P_d$ and let $\mathrm{Crit}(f)$ denote the set of critical points of $f$. We say $f$ is {\em renormalizable at $c$} for some $c\in \mathrm{Crit}(f)$ if there exist quasidisks $c\in \Omega\Subset \Omega'$ and $p\ge 1$ such that
\begin{itemize}
\item $f^p|_{\Omega}:\Omega\to \Omega'$ is a polynomial-like map with connected Julia set;
\item for any $c'\in\mathrm{Crit}(f)$,  there is at most one $0\le j<p-1$ such that  $c'\in f^j(\Omega)$;
\item $p>1$ or $\mathrm{Crit}(f) \not\subset \Omega$.
\end{itemize}

The Douady-Hubbard Straightening Theorem says that each polynomial-like map $g$ with connected Julia set is hybrid to a unique polynomial up to an affine conjugacy. To determine the straightening uniquely, it is convenient to introduce an external marking for $g$.
\begin{definition}[Access and external marking]Let $g: U \to V$ be a polynomial-like map connected filled Julia set. A path to $K(g)$ is a continuous map
$\gamma: [0, 1] \to U$ such that $\gamma ((0, 1]) \subset U \setminus K(g)$ and $\gamma(0) \in J(g)$. We say two paths $\gamma_0$ and $\gamma_1$
to $K(g)$ are homotopic if there exists a continuous map $\tilde \gamma : [0, 1] \times [0, 1] \to U$ such that
\begin{enumerate}
\item $t \mapsto \tilde \gamma (s, t)$ is a path to $K(g)$ for all $s \in [0, 1]$;
\item $\tilde \gamma (0, t) = \gamma_0(t)$ and $\tilde\gamma (1, t) = \gamma_1(t)$ for all $t \in [0, 1]$;
\item $\tilde \gamma (s, 0) = \gamma_0(0)$ for all $s \in [0, 1]$.
\end{enumerate}
An access to $K(g)$ is a homotopy class of paths to $K(g)$.\par
An external marking of $f$ is an access $\Gamma$ to $K(g)$ which is forward
invariant in the following sense. For every representative $\gamma$ of $\Gamma$, the connected component of $g(\gamma)\cap U$ which intersects $J(f)$ is also a representative of $\Gamma$.
\end{definition}
The {\em standard external marking} for a polynomial $f$ with connected Julia set is defined as the homotopy class of $\mathcal R_f(0)$ in the sense of paths to $K(f)$.
 
 Recall that $f_0 \in \mathcal C_{d_1+d_2}$ is a primitive polynomial with a periodic critical point $0$ of multiplicity $d_1-1$ and a free Julia critical point $c_0$ of multiplicity $d_2$. Let $U_0$ be the Fatou component of $f_0$ which contains the superattracting periodic point $0$. It is well-known that $\partial U_0$ is a Jordan curve (see~\cite{RY}). Let $m_0$ be the period of $0$ under $f_0$. Note that $f^{m_0}|_{U_0}$ has a polynomial-like extension $F_0:U\to V$ such that $K(F_0)=\overline{U_0}$. 
 
A polynomial $f\in\mathcal C(\lambda_{f_0})$ is said to be {\em $\lambda_{f_0}$-renormalizable} if there exist quasidisks $U'\Subset V'$ such that $F=f^{m_0}|_{U'}:U'\to V'$ is a polynomial-like map with the filled Julia set $$K(F)=K(U_0,f):=\bigcap\limits_{\theta\sim_{\lambda_{f_0}}\theta'} S(\theta,\theta', U_0)\cap K(f),$$
where $S(\theta,\theta',U_0)$ is the component of $\C\setminus \overline{\mathcal R_{f}(\theta)\cup \mathcal R_{f}(\theta')}$ which contains external rays landing on $\partial U_0$. Such a polynomial-like $F$ is called a {\em $\lambda_{f_0}$-renormalization for $f$}. Note that a $\lambda_{f_0}$-renormalization $F$ for $f$ is unique up to the hybrid conjugacy.

 To define a canonical external marking for the $\lambda_{f_0}$-renormalization $F$, we need the definition of internal angle system which was introduced by Inou-Kiwi~\cite{IK}. An {\em internal angle system for $f_0$} is a homeomorphism $\alpha:\partial U_0\to \R/\Z$ such that
 $\alpha\circ f^{m_0}=m_{d_1}\circ \alpha \mod 1$ on $\partial U_0$.
Such a homeomorphism is determined by $\alpha^{-1}(0)$ which is a $f^{m_0}_0$-fixed point. Choose an external angle $\theta_0$ such that $\mathcal R_{f_0}(\theta_0)$ land at $\alpha^{-1}(0)$. The {\em external marking for $F$ induced by $\alpha$} is defined as the homotopy class of the component of $\overline{\mathcal R_{f}}(\theta)\cap U$ which intersects $K(F)$ in the sense of paths to $K(F)$.
The homotopy class does not depend on the choice of $\theta_0$ (see~\cite[Remark~3.16]{IK}), so it is well defined.
\begin{definition}[Tuning]
Let $f_0$ and an internal angle system $\alpha$ be given. We say $f_0$ can be {\em tuned with a polynomial $g\in \mathcal C_{d_1}$}  if there exists a polynomial $f\in \mathcal C(\lambda_{f_0})$ and quasidisks $U'\Subset V'$ with the following properties:
\begin{itemize}
\item $F=f^{m_0}|_{U'}:U'\to V'$ is a $\lambda_{f_0}$-renormalization for $f$ which is hybrid equivalent to $g$;
\item there exists a corresponding hybrid conjugacy $\varphi$ between $F$ and $g$ which respects the external markings, that is, $\varphi$ maps the external marking for $F$ induced by $\alpha$ to the standard external marking for $g$.
\end{itemize}
Such an $f$ is called a {\em tuning} of $f_0$ by $g$.\par
\end{definition}
Clearly, $f_0$ can be tuned with  $z\mapsto z^{d_1}$ and $f_0$ itself is a tuning of $f_0$ by $z\mapsto z^{d_1}$. In this paper, we show that $f_0$ can be tuned with any $g\in \mathcal C_{d_1}$:
\begin{Theorem B} Let $f_0\in \mathcal C_{d_1+d_2}$ be a primitive  polynomial with a periodic critical point $0$ of multiplicity $d_1-1$ and a Julia critical point $c_0$ of multiplicity $d_2$. Fix an internal angle system $\alpha$ for $f_0$. If $f_0$ has no neutral periodic points  and $f_0$ is non-renormalizable at $c_0$, then $f_0$ can be tuned with any $g\in \mathcal C_{d_1}$.
\end{Theorem B}
The proof of Theorem~B will be given in section~\ref{sec:qc surgery}.

\section{Yoccoz puzzle}\label{sec:puzzle}
From this section to the end of this paper, we assume $f_0$ is not renormalizable at $c_0$ and has no neutral periodic points.
A finite set $Z$ is said to be {\em $f_0$-admissible} if the following hold:
\begin{itemize}
\item $f_0(Z)\subset Z$,
\item each periodic point in $Z$ is repelling and does not lie in the orbit of the critical point $c_0$;
\item for each $z\in Z$,  there exist at least two external rays landing at $z$.
\end{itemize}
Note that the definition differs slightly from that of \cite{SW} as we require that the periodic points in $Z$ do not lie in the orbit of the free critical point. This condition is used to guarantee the free critical point is not contained in the boundary of any Yoccoz puzzle pieces which are defined as following. 

Let $\Gamma_0=\Gamma^Z_0$ denote the union of all the external rays landing on $Z$, the set $Z$ itself and the equipotential curve of height $1$. 
For each $n\ge 1$, define $\Gamma^Z_n=f_0^{-n}(\Gamma^Z_0)$.
A bounded component of $\C\setminus \Gamma^Z_n$ is called a {\em $Z$-puzzle piece} of depth $n$. Let $Y^{(n)}_Z=Y^{(n)}_Z(0)$ and $X^{(n)}_Z=Y^{(n)}_Z(c_0)$ denote the puzzle pieces of depth $n$ which contains $0$ and $c_0$ respectively.

The aim of this section is to prove the following:

\begin{Theorem}\label{thm:puzzle} Let $f_0\in \mathcal C_{d_1+d_2}$ be a primitive  polynomial with a periodic critical point $0$ of multiplicity $d_1-1$ and a Julia critical point $c_0$ of multiplicity $d_2$.  If $f_0$ has no neutral cycles and $f_0$ is non-renormalizable at $c_0$,  then there exists a buried biaccessible repelling periodic point $\tau_0$ with the following properties:
\begin{enumerate}
 \item $Z_0:=\mathrm{orb}(\tau_0)$ is an $f_0$-admissible  finite set;
 \item $\bigcap_{n=0}^\infty Y^{(n)}_{Z_0}=\overline{U_0}$, where $U_0$ is Fatou component of $f_0$ which contains $0$;
 \item $\bigcap_{n=0}^\infty X^{(n)}_{Z_0}=\{c_0\}$;
 
\end{enumerate}
\end{Theorem}

The proof is slightly different from that of ~\cite[Theorem~3.1]{SW} since $f_0$ has a free Julia critical point in this case. However, the methods of the proofs are similar in spirit. To make this section self-contained, we give a proof of Theorem~\ref{thm:puzzle} here. The readers who are familiar with the results and proofs in~\cite{SW} may skip the proof of Theorem~\ref{thm:puzzle}.

We say a point $w\in J_{f_0}$ is {\em bi-accessible} if it is the common landing points of two distinct external rays.
A point in $J_{f_0}$ is called {\em buried} if it is not in the boundary of any bounded Fatou component.

\begin{Lemma}[{\cite[Lemma~3.1]{SW}}]\label{lem:puuzle-1} Let $f$ be a polynomial with connected Julia set.  Then any bi-accessible point in the boundary of a bounded Fatou component is eventually periodic.
\end{Lemma}
\begin{proof} See the proof of {\cite[Lemma~3.1]{SW}}.
\end{proof}

\begin{Lemma}\label{lem:puzzle0}
Under the assumption in Theorem~\ref{thm:puzzle}, $f_0$ has a bi-accessible repelling periodic point on $\partial U_0$.
\end{Lemma}
\begin{proof}
Since all the periodic point of $f_0$ are non-neutral and $f_0$ is not renormalizable at $c_0$, the Julia set $J_{f_0}$ of $f_0$ is locally connected (\cite[Theorem~1.1]{KS}). 
Then the lemma follows by the same method as in the proof of \cite[Lemma~3.2]{SW}.

\end{proof}

\begin{Lemma}\label{lem:period}Let $f\in \mathcal C_{d'}$ for some $d'\ge 2$ and let $U$ be a fixed Fatou component of $f$. If $p$ is a fixed point lying on $\partial U$ and $\mathcal R_f(t)$ lands at $p$, then $d't=t \mod 1$.
\end{Lemma}
\begin{proof} Let $\Theta$ denote the set of the arguments $\theta$ for which $\mathcal R_f(\theta)$ lands at $p$. Note that $m_{d'}$ maps $\Theta$ onto itself and preserves cyclic order, where $m_{d'}:\R/\Z\to \R/\Z$, $x\mapsto d'x \mod 1$. Thus all the arguments $\theta\in \Theta$ have the same period under $m_d'$. It suffices to show there exists $\theta_0 \in \Theta$ such that $m_{d'}(\theta_0)=\theta_0$.  If $\Theta$ is a singleton, then $d't=t\mod 1$ since both $t$ and $m_{d'}(t)$ belong to $\Theta$. Now we assume $\# \Theta\ge 2$. Let $\Gamma$ be the closure of the union of all the external rays with arguments $\theta \in \Theta$. Let $\Omega$ be the component of $\C\setminus \Gamma$ which contains $U$ and let $\Omega'$ be the component of $\C\setminus f^{-1}(\Gamma)$ which contains $U$. Then $\partial \Omega \subset \partial \Omega'$ and $f:\Omega' \to \Omega$ is a holomorphic proper map. It follows that $f$ must fix the boundary of $\Omega$. In other words, the external rays in $\partial V$ are fixed under $f$.

\end{proof}

\begin{Lemma}\label{lem:puzzle} Under the assumption in Theorem~\ref{thm:puzzle},  $f_0$ has a buried biaccessible repelling periodic point which does not lie in the orbit of the critical point $c_0$.
\end{Lemma}
\begin{proof} 
The proof is based on the pigeonhole principle. 
By Lemma~\ref{lem:puzzle0}, there is a bi-accessible repelling periodic point $p$ on $\partial U_0$.  Since $f_0$ is primitive, the period of $p$ must be a multiple of the period of $U_0$. Without loss of generality, we assume $U_0$ is fixed by $f_0$. By Lemma~\ref{lem:period}, the ray period of $p$ is the same as the period of $p$. In other words, if the period of $p$ is $\ell$, then each external ray landing at $p$ has external angle $\theta$ with $d^{\ell}\theta=\theta \mod 1$, where $d=d_1+d_2$.

Note that the polynomial $f^\ell_0$ has $d^\ell$ fixed point and $0$ is its unique attracting fixed point. As  $f_0$ has no neutral periodic points, we know all the fixed points of $f^{\ell}_0$ are repelling except $0$.  Since $p$ is bi-accessible, by Lemma~\ref{lem:period},  there are at least two $f^\ell_0$-fixed external  rays landing at $p$. Thus there are at most $(d^{\ell}-3)$ $f^\ell_0$-fixed external rays landing at other  $d^{\ell}-2$  repelling fixed points. Since each $f^\ell_0$-fixed external ray must land at a repelling $f^{\ell}_0$-fixed point, there exists a repelling fixed point $\widehat\alpha$ of $f^\ell_0$ at which there are no $f^\ell_0$-fixed external rays landing. Thus, there are at least two external rays (which are not $f^\ell_0$-fixed) landing at $\widehat\alpha$.  It follows $\widehat\alpha\notin \partial U_0$  from Lemma~\ref{lem:period}. Since $f_0$ is primitive, and so is $f^{\ell}_0$, $\widehat\alpha$  cannot lie on the boundary of any preimage of $U_0$ under $f^{\ell}_0$. Note that all the Fatou component of $f^{\ell}_0$ is a preimage of $U_0$ under $f^{\ell}_0$. We conclude that  $\widehat\alpha$ is a buried repelling fixed point of $f^\ell_0$. If $\widehat\alpha$ does not lie in the orbit of $c_0$, then we are done.\par
Now assume $\widehat\alpha \in \mathrm{orb}(c_0)$.  Let $\ell'>\ell$ denote the ray period of $\widehat\alpha$.  Let $P(\ell')$ denote the set of all the repelling periodic points with period $\ell'$.  For any $z\in P(\ell')$, there is at least one external ray with period divisible by $\ell'$ landing at $z$.  As there are at least two external rays of period $\ell'$ landing at $\widehat\alpha \notin P(\ell')$, there are at most $d^{\ell'}-d^{\ell'-1}-1$ external rays with period $\ell'$ possible to land on $P(\ell')$.  Since $\# P(\ell')=d^{\ell'}-d^{\ell'-1}$, there exists an $\widehat\alpha' \in P(\ell')$ such that there are no $f^{\ell'}_0$-fixed external rays landing at $\widehat\alpha'$.  So $\widehat\alpha'$ is a biaccessible repelling $f^{\ell'}_0$-fixed point. Again by Lemma~\ref{lem:period} and the primitive property of $f_0$, we can conclude that $\widehat\alpha'$ is buried. 
Since $\widehat\alpha \in \mathrm{orb}(c_0)$ and $\ell'>\ell$, $\widehat\alpha'$ cannot lie in the critical orbit $\mathrm{orb}(c_0)$. 

\end{proof}

\newcommand{\Crit}{\text{Crit}}
\begin{proof}[Proof of (1) and (2) of Theorem~\ref{thm:puzzle}]
Recall that $m_0$ is the period of $0$.
By Lemma~\ref{lem:puzzle}, $f_0$ has a buried biaccessible repelling periodic point $\xi_1$ such that $\xi_1 \notin \mathrm{orb}(c_0)$. We use $Z_1=\mathrm{orb}(\xi_1)$ to make Yoccoz puzzle. Let $s_1\le m_0$ be the smallest positive integer  such that $f^{s_1}_0(0 ) \in \bigcap_{n=0}^{\infty} Y^{(n)}_{Z_1}$. Then there exist positive integers $q_1$ and $u$ such that $f^{s_1}_0:Y^{(n+s_1)}_{Z_1} \to Y^{(n)}_{Z_1}$ is a holomorphic proper map of  degree $u$ for all $n\ge q_1$. We claim $u=d_1$. Indeed, if $u>d_1$ then there exists $0\le t_1<s_1$ such that $c_0 \in f^{t_1}_0(\bigcap_{n=0}^{\infty} Y^{(n)}_{Z_1})$. This is impossible since $f_0$ is non-renormalizable at $c_0$.

By the thickening technique (\cite{Mil4}), $f^{s_1}_0:Y^{(q_1+s_1)}_{Z_1} \to Y^{(q_1)}_{Z_1}$ can extend to a polynomial-like map $g_1$ with filled Julia 
set $K(g_1)=\bigcap_{n=0}^{\infty} Y^{(n)}_{Z_1}$. By Douady-Hubbard Straightening Theorem, $g_1$ is hybrid to a primitive  postcritically finite hyperbolic polynomial $Q_1$ of degree $d_1$. If $s_1=m_0$, then $Q_1=z^{d_1}$ and so $\bigcap_{n=0}^{\infty} Y^{(n)}_{Z_1}= \overline{U_0}$.  Then we take $\tau_0=\xi_1$ and we are done. If $s_1\ne m_0$, then $Q_1\ne z^{d_1}$. It follows \cite[Lemma~3.3]{SW} that $Q_1$ has a buried biaccessible repelling periodic point $\zeta_1$. Let $h_1$ denote the hybrid conjugacy between $g_1$ and $Q_1$. Then $\xi_2=h^{-1}_1(\zeta_1)$ is accessible along at least two curves (the preimage of two external rays of $Q_1$ under $h^{-1}_1$)  which are not homotopy equivalent relative to $\C\setminus K(f_0)$. By \cite[Theorem~2]{LP},  there are at least two $f_0$-external rays landing at $\xi_2$. So $\xi_2$ is a buried biaccessible repelling periodic point of $f_0$. By using a similar argument in the proof of Lemma~\ref{lem:puzzle}, we can further assume $\xi_2\notin \mathrm{orb}(c_0)$. 

Now  let $Z_2=\mathrm{orb}(\xi_2)$ and construct Yoccoz puzzle. Let $s_2$ be the smallest positive integer such that $f^{s_2}_0(0) \in \bigcap_{n=0}^{\infty} Y^{(n)}_{Z_2}$. Note that $s_2>s_1$. Indeed, the external rays landing at $\xi$ separate $0$ from $f_0(\bigcap_{n=0}^{\infty} Y^{(n)}_{Z_1}),\ldots, f_0^{s_1-1}(\bigcap_{n=0}^{\infty} Y^{(n)}_{Z_1})$ and $f^{s_1}_0(0)$. 

Inductively, we can obtain a strictly increasing sequence $\{s_k\}$ and a sequence $\{Z_k=\mathrm{orb}(\xi_k)\}$ with the following properties:

\begin{itemize}
\item $\xi_k$ is a buried biaccessible repelling periodic point of $f_0$ which lies outside the orbit of $c_0$;
\item $s_k\le m_0$ is the smallest positive integer  such that $f^{s_k}_0(0) \in \bigcap_{n=0}^{\infty} Y^{(n)}_{Z_k}$. 
\end{itemize}

Since $s_k$ is a strictly increasing sequence of positive integers bounded by $m_0$, there is a first moment $k_0$ such that $s_k=m_0$. Take $\tau_0=\xi_k$ and we are done. Indeed, $f^{m_0}_0: Y^{(m_0+n)}_{Z_k} \to Y^{(n)}_{Z_k}$ has a polynomial-like extension which is hybrid to $z\mapsto z^{d_1}$ for all sufficiently large $n$. Thus the filled Julia set of $f^{m_0}_0|_{Y^{(m_0+n)}}$ must equal to $\overline{U_0}$.

\end{proof}

From now on, we will fix an $f_0$-admissible set $Z_0=\mathrm{orb}({\tau_0})$ which is given by the above proof. For each $n\in \mathbb N$, let $Y^{(n)}$ and $X^{(n)}$ denote the critical puzzle pieces $Y^{(n)}_{Z_0}(0)$ and $Y^{(n)}_{Z_0}(c_0)$ respectively.  We complete the proof of Theorem~\ref{thm:puzzle} by showing that $X^{(n)}$ shrinks to a point:

\begin{Lemma}\label{lem:X shrinking} $\bigcap_{n\in \mathbb N} X^{(n)}=\{c_0\}$.
\end{Lemma}
We say $c_0$ is {\em combinatorially recurrent} if for any $n\in\mathbb N$, there exists $j\ge1$ such that $f^j_0(c_0) \in X^{(n)}$.
The proof for the non-combinatorially recurrent case is simple, and is treated in the same way as in the quadratic case. See ~\cite[Lemma~1.8,Theorem~1.9]{Mil4} for an example. To deal with the combinatorially recurrent case, we need the following definition.
\begin{definition}
 A generalized polynomial-like map is a holomorphic proper map $g:U\to V$ with the following properties:
\begin{itemize}
\item $V$ is a quasidisk;
\item each component of $U$ is a quasidisk contained compactly in $V$;
\item $g$ has finitely many critical points.
\end{itemize}
\end{definition}
For more details for generalized polynomial-like map, we refer the readers to ~\cite{KL}.
Now we use the Yoccoz puzzle induced by $Z_0$ to associate a generalized polynomial-like map to $f_0$.

\begin{definition} An {\em APL } map $g:U\to V$ is a holomorphic proper map which satisfies the following :
$U\subset V$ and $K(g):=\bigcap_n g^{-n}(U) $ is compactly contained in  $U$.
\end{definition}
 It follows from \cite[Lemma~2.4]{LY} that every APL map has a polynomial-like restriction with the filled Julia set  equal to $K(g)$.
 
\begin{Lemma}\label{lem:non-degenerated} If $c_0$ is combinatorially recurrent, then there exists a positive integer $\nu_0$ such that the component of the domain of the first return map to $X^{(\nu_0)}$ which contains $c_0$ is compactly contained $X^{(\nu_0)}$.
\end{Lemma}
\begin{proof} By the proof of (1) and (2) of Theorem~\ref{thm:puzzle}, there exists $\kappa_0>0$ large enough so that the following holds:
 \begin{itemize}
 \item $\overline {Y^{(\kappa_0)}},\overline{f_0(Y^{(\kappa_0)})},\ldots, \overline{f^{m_0-1}_0(Y^{(\kappa_0)})}$ are disjoint;
 \item $f^{m_0}_0:Y^{(\kappa_0)} \to Y^{(\kappa_0-m_0)}$ is an APL map with filled Julia set $\overline{U_0}$.
  \end{itemize}
  Let $P_0=X^{(\kappa_0)}$. For $n\ge 1$, we define inductively $P_{n}$ as the component of the domain of first return map to $P_{n-1}$ which contains $c_0$. For all $n\in\mathbb N$, the first return map $R_{n}|_{P_{n+1}}:P_{n+1} \to P_{n}$ is $d_2$-to-$1$. Note that there exists a first moment $s_0$ such that $P_{s_0+1} \Subset P_{s_0}$. It is well known since $f_0$ is non-renormalizable at $c_0$, see the proof of ~\cite[Lemma~2.2]{KS} for an example. Let $X^{(\nu_0)}=P_{s_0}$ and we are done.
\end{proof}

\begin{proof}[Proof of Lemma~\ref{lem:X shrinking}] 
Assume $c_0$ is combinatorially recurrent  and let $X^{(\nu_0)}=P_{s_0}$,  $P_{n}~(n\in\mathbb N)$ be as in the proof Lemma~\ref{lem:non-degenerated}. It suffices to show $P_n$ shrinks to $c_0$.  Let $V=P_{s_0+1}$ and let $R:U\to V$ be the first return map to $V$ under $f_0$. Note that $R:U\to V$ is a generalized polynomial map with a unique critical point $c_0$.  Since $f_0$ is non-renormalizable at $-c_0$,  so is $R:U\to V$. This implies the diameter of the puzzle piece $X^{(n)}$ shrinks to $0$ and $J_{f_0}$ is locally connected at $c_0$, see~\cite{KL} and \cite[Theorem~1.5]{KS}.
\end{proof}

Every puzzle piece which attaches $Z_0$ also shrinks to a point:
\begin{Proposition}\label{prop:noncrpiece}Let $z\in \bigcup_{j=0}^{\infty}f_0^{-j}(Z_0)$. Then
$$\sup\{\mathrm{diam}(Y)\mid Y \text{ is a puzzle piece of depth }n ~\text{such that}~z\in \overline{Y}\}\to 0\text{ as } n\to\infty.$$
\end{Proposition}

 \begin{proof} According to Theorem~\ref{thm:puzzle}, there exists $n_0$ large such that $Y\cap (X^{(n_0)}\cup Y^{(n_0)})=\emptyset$. Then the proposition follows by \cite[Lemma~1.8]{Mil4}.
\end{proof}
For a more general result by Kozlovski and van Strien, see \cite[Theorem~1.1]{KS}. The following corollary is a combinatorial version for Proposition~\ref{prop:noncrpiece}:

\begin{Coro}\label{coro:shrinking}Let $z \in \bigcup_{j\ge 0}f^{-j}_0(Z_0)$ and let $\Theta(z)$ be the set of angles $\theta$ for which $\mathcal R_{f_0}(\theta)$ lands at $z$. Then for each component $(\theta_1,\theta_2)$ of $\mathbb R/\mathbb Z\setminus \Theta(z)$, there exist two monotone sequences $\{a_n\}$ and $\{b_n\}$ in $(\theta_1,\theta_2)$ such that $a_n \sim_{\lambda_{f_0}} b_n$ for all $n>0$ and $a_n \to \theta_1$ and $ b_n \to \theta_2$ as $n$ tends to infinity. 
\end{Coro}

As an application, we prove:
\begin{Lemma}\label{lem:nondegenerate}Suppose two external rays  $\mathcal R_{f_0}(t)$ and $\mathcal R_{f_0}(s)$ of $f_0$ land at $z_s$ and $z_t$ respectively. If $z_t,z_s \in \bigcup\limits_{n=0}^{\infty} f^{-n}_0(Z_0)$ and $z_t \ne z_s$, then $\mathcal R_f(t)$ and $\mathcal R_f(s)$ land at two different points for every $f \in \mathcal C(\lambda_{f_0})$.  
\end{Lemma}
\begin{proof} Let $\Theta(z_t)$ and $\Theta(z_s)\subset \mathbb R/\mathbb Z$ be the set of landing angles for $z_t$ and $z_s$ under $f_0$ respectively. There are two component $C$ and $D$ of $\mathbb R/\mathbb Z \setminus (\Theta(z_t)\cup \Theta(z_s))$ such that $C=(\theta_{\max}(s),\theta_{\min}(t))$ and $D=(\theta_{\max}(t),\theta_{\min}(s))$ where $\theta_{\min}(t),\theta_{\max}(t) \in \Theta(z_t)$ and $\theta_{\min}(s),\theta_{\max}(s) \in \Theta(z_s)$. By Corollary~\ref{coro:shrinking}, there exist $\{a_n\}\subset C$ and $\{b_n\}\subset D$ such that $a_n \sim_{\lambda_{f_0}}b_n$ and  $a_n \to \theta_{\min}(t)$, $b_n \to \theta_{\max}(t)$ as $n \to \infty$. Now assume $\mathcal R_f(t)$ and $\mathcal R_f(s)$  land at a common point $w$. Since $\lambda_{f_0} \subset \lambda_f$, $\mathcal R_f(a_n)$ and $\mathcal R_f(b_n)$ land at a common point. Moreover, this point must be $w$  by virtue of the construction. Hence, there are infinitely many external rays of $f$ landing at $w$, which is a contradiction.
\end{proof}
The above lemma implies all the $f\in \mathcal C(\lambda_{f_0})$ have the same puzzle structure with respect to $Z_0$.

$$
{\beginpicture
\setcoordinatesystem units <.8cm,.8cm>

\circulararc 360 degrees from 3 3 center at 0 0
\plot -2 0  -3 3 /
\plot -2 0  -3 -3 /
\plot -2 0  -4.25 0 /
\plot -1 0  -1.9 3.8 /

\plot 2 0  3 3 /
\plot 2 0  3 -3 /
\plot 2 0  4.25 0 /
\plot -1 0  -1.9 -3.8 /

\put{$C$} at 0 4.6
\put{$\theta_{\min}(t)$} at -3.3 3.5 
\put{$\theta_{\max}(t)$} at -3.3 -3.5
\put{$t$} at -4.5 0
\put{$a_n$} at -1.9 4.1

\put{$\theta_{\max}(s)$} at 3.4 3.5 
\put{$\theta_{\min}(s)$} at 3.4 -3.5
\put{$s$} at 4.5 0
\put{$b_n$} at -1.9 -4.1
\put{$D$} at 0 -4.6

\put{figure~3.1} at 0 -6
\endpicture}
$$


\section{Kahn's quasiconformal distortion bounds}\label{sec:Kahn}
From this section to the end of this paper, we fix an $f_0$-admissible set $Z_0=orb(\tau_0)$  which is given by Theorem~\ref{thm:puzzle}. 

In this section, we will modify the argument in~\cite{SW} to obtain a K-qc extension principle. This enables us to use the criterion for convergence of Thurston's algorithm which was established in \cite{SW}.

Let $L_n$ denote the domain of the first landing map to $Y^{(n)}$:
\begin{equation}\label{eqn:dfnLn}
L_n=\left\{z\in \C: \exists k\ge 0 \text{ such that } f_0^k(z)\in Y^{(n)}\right\}.
\end{equation}
Our main goal in the section is to show:
\begin{Theorem}\label{thm:Kqcpuzzle'} For any puzzle piece $Y$, there exists a positive integer $N$ such that 
\[\mathcal{QD}(L_{N}\cap Y, Y)<\infty.\]
\end{Theorem}
As in \cite{SW}, the difficulty in proving the above theorem is that the landing domains $L_{N}$ may come arbitrarily close to the boundary of $Y$. To deal with the situation, we shall need a toy model developed by Kahn (\cite{Kahn}). We first construct an expanding map which is quasiconformally conjugate to Kahn's recursively notched square model, which will imply Theorem~\ref{thm:Kqcpuzzle'} holds for all the puzzle pieces of depth $0$. To complete the proof, we use a pullback argument. The proof is adapted from ~\cite[Section~4]{SW}. The readers who are familiar with that of \cite{SW} may skip the proof. The main difference is the following. In \cite{SW}, $\partial Y'$ is disjoint from the post-critical set for any puzzle piece $Y'$ since $f_0$ is post-critically finite. Thus for each $z\in J_{f_0} \cap Y'$, we can pullback a definite small neighborhood of $z$ univalently to an arbitrary depth. In our case, $f_0$ has a free Julia critical point $c_0$ which may accumulate on the boundaries of all the puzzle pieces. Thus we must choose a neighborhood of $\tau_0$ carefully so that we can pullback it univalently to some given depth.  To make it self-contained, we give an independent proof in this section.

\subsection{Quasiconformal Distortion Bounds}

Let us first recall some terminology from \cite{Kahn}. Let $U \subset \mathbb C$ be a Jordan domain and $A$ be a measurable subset of $U$. We say that $(A,U)$ has {\em bounded qc distortion} if there exists a constant $K\ge 1$ with the following property: if $\varphi:U\to \varphi(U)$ is a quasiconformal map and $\bar\partial \varphi=0$ a.e. outside $A$, then there is a $K$-q.c map $\tilde \varphi:U \to \varphi(U)$ such that $\tilde \varphi=\varphi$ on $\partial U$. Let  $\mathcal{QD}(A,U)$ denote the smallest $K$ satisfying the property.

We shall need the following easy facts.
\begin{Lemma}\label{compact}\cite[Fact~1.3.6]{Kahn} If $A \subset U$ is compact, then $\mathcal{QD}(A,U)<\infty$.
\end{Lemma}

\begin{Lemma}\label{qc1}\cite[Fact~1.3.4]{Kahn} Let $U$ and $V$ be Jordan domains in $\mathbb C$ and let $A$ be a measurable subset of $U$.
If there exists a $L$-qc map $g:U\to V$ and $\mathcal{QD}(A,U)<\infty$, then \[\mathcal{QD}(g(A),V)\le L^2\mathcal{QD}(A,U).\]
\end{Lemma}

\begin{Lemma}\label{qc2} The following statements are equivalent:
\begin{enumerate}
\item [(i)] $\mathcal{QD}(A,U)=C<\infty$;
\item [(ii)] For any qc map $\varphi:U \to \varphi(U)$, if $\mathrm{Dil}(\varphi)\le K$ for some $K\ge 1$ a.e. outside $A$, then there is a $KC$-qc map $\widetilde\varphi:U \to \varphi(U)$ such that $\widetilde \varphi=\varphi$ on $\partial U$.
\end{enumerate}
\end{Lemma}
See~\cite[Lemma~4.3]{SW} for a proof.
\begin{Lemma}\label{lem:qc3} Let $U$ be a Jordan domain and let $A$ be a measurable subset of $U$. Assume there is an open subset $W$ of $U$  with the following properties:
\begin{itemize}
\item $W$ is a disjoint union of finitely many Jordan domains;
\item for each component $T$ of $W$, $\partial T\cap A=\emptyset$ and $\mathcal{QD}(A\cap T,T)<\infty$;
\item $A\setminus W$ is compactly contained in $U$.
\end{itemize} 
Then $\mathcal{QD}(A,U)<\infty$.
\end{Lemma}

\begin{proof}
Let $\phi:U\to \phi(U)$ be a quasiconformal map with $\bar\partial \phi=0$ a.e. outside $A$. 
For each component  $T$ of $W$, there exists a $K_T$-qc map $\phi_T:T\to \phi(T)$ such that $\phi_T|_{\partial T}=\phi|_{\partial T}$ since  $\mathcal{QD}(A\cap T,T)<\infty$. Let $K=\max\limits_T K_T$. Replace $\phi$ by $\phi_T$ on each component $T$, we  get a qc map $\varphi$ which is $K$-qc a.e. outside $A\setminus W$ and $\varphi|_{\partial U}=\phi|_{\partial U}$. Since $A\setminus W$ is compactly contained in $U$, $\mathcal{QD}(A\setminus W, U)=K'<\infty$. By Lemma~\ref{qc2}, there exists a $KK'$-qc map with the same boundary value as $\phi$.
\end{proof}

We shall now recall the {\em recursively notched square model} developed in~\cite{Kahn}.
Let $S=(0,1)\times (-1/2,1/2)$. Let $\mathcal{I}$ denote the collection of the components of $(0,1)\setminus \mathcal{C}$, where $\mathcal{C}$ is the ternary Cantor set. Let
\begin{equation}\label{eqn:setN}
\mathcal{N}=\bigcup_{I\in\mathcal{I}} \overline{I}\times [-|I|/2, |I|/2]
\end{equation}
which is a countable disjoint union of closed squares.
The following is ~\cite[Lemma 2.1.1]{Kahn}:
\begin{Theorem}\label{thm:Kahn}
$\mathcal{QD}(\mathcal{N}, S)<\infty.$
\end{Theorem}


\subsection{Reduce to the toy model}
In the following we write $\mathcal{R}(\theta)$ for $\mathcal{R}_{f_0}(\theta)$. A {\em combinatorial ray-pair} is, by definition, a pair $(\theta,\tilde\theta)$ such that $\theta\ne\tilde\theta$ and $\theta\sim \tilde\theta$. The union  $\overline{\mathcal R(\theta)}\cup \overline{\mathcal R(\tilde\theta)}$ is called the {\em geometric ray-pair} with respect to $(\theta,\tilde\theta)$ and we denote it by $\mathcal R(\theta,\tilde\theta)$.
By a {\em slice}  we mean  an open set $S(\theta_1,\theta_2,\tilde\theta_2,\tilde\theta_1)$ bounded by two disjoint geometric ray-pairs $\mathcal R(\theta_i,\tilde\theta_i)$, $i=1,2$ such that $\theta_1,\theta_2,\tilde\theta_2,\tilde\theta_1$ lie in $\mathbb R/\mathbb Z$ in the anticlockwise order.

Let $\tau_0$ be given by Theorem~\ref{thm:puzzle}, the external rays landing at $\tau_0$ cut the complex plane $\mathbb C$ into finitely many sectors $S_0,\cdots,S_{q-1}$  where $q$ is the number of external rays landing at $\tau_0$. Without loss of generality, we assume $0\in S_0$.  The boundary of  $S_0$ is a geometric ray-pair:
there exists $\theta^-,\theta^+\in \mathbb R/\mathbb Z$ such that $\partial S_0=\mathcal R(\theta^-,\theta^+).$ By Proposition~\ref{prop:noncrpiece}, for $n$ sufficiently large, there exists a unique puzzle piece $V_n$ lying in $S_0$ such that  $\tau_0\in  \partial V_n$ and $0\notin V_n$. The boundary of $V_n$ consists of portions of an equipotential curve and a finite set of geometric ray-pairs, cut off at that potential height. Among these geometric ray-pairs, there exists a unique ray-pair $\mathcal R(t_n,s_n)$ which separates $\tau_0$ from $0$. Let ${S}(n)$ be the slice bounded by $\mathcal{R}(t_n, s_n)$ and $\partial S_0$. Let $\widehat{S}(n)=\{z\in {S}(n): G_{f_0}(z)<(d_1+d_2)^{-n}\}$.  Without loss of generality, we assume $\theta^-,t_n,s_n,\theta^+$ lie in $\mathbb R/\mathbb Z$ in the anticlockwise order. The following lemma is a consequence of Proposition~\ref{prop:noncrpiece}.
\begin{Lemma}\label{lem:slice}$\text{diam}(\widehat{S}(n))\to 0$ as $n\to \infty$.
\end{Lemma}


Let $p$ be the period of $\theta^+$. By Lemma~\ref{lem:slice},  we can choose $N_1$ sufficiently large so that the closure of $\widehat S:=\widehat S(N_1)$ is disjoint from the set $\bigcup_{j=0}^{p} f^j_0(\{0,c_0\})$. Thus we can define a single-valued univalent branch $g$ of $f^{-p}_0$ on a neighborhood of $\widehat S(N_1)$ such that $g(\tau_0)=\tau_0$.

Let $\tau'_0$ be the landing point of the external ray $\mathcal R(t_{N_1})$ and let $r$ be the first moment for which $f^r_0(R(t_{N_1},s_{N_1}))=\partial S_0$. Note that $g^k(\widehat S(N_1)) \subset \widehat S(N_1+kp)$ for all $k\in \mathbb N$. By  Lemma~\ref{lem:slice},  the closure of $\widehat S(N_1+kp)$ is disjoint from the set $\bigcup_{j=0}^{r} f^j_0(\{0,c_0\}$ for all $k$ large. Thus we can define a single-valued univalent branch $h_k$ of $f^{-r}_0$ on a neighborhood of $\widehat S(N_1+kp)$ such that $h_k(\tau_0)=\tau'_0$ for all $k$ large. Let $\widehat A=g(\widehat S(N_1))$ and $\widehat B_k=h_k\circ g^{kp}(\widehat S(N_1))$. Note that $A$ and $B_k$ are quasidisks. It follows from  Lemma~\ref{lem:slice} that there exists $k_0$ such that $\overline {\widehat A}\cap \overline{\widehat B_k}=\emptyset$ for all $k\ge k_0$.  Let $\widehat B=\widehat B_{k_0}$ and $h=h_{k_0}$. Define a map $\widehat F:\widehat A\cup \widehat B \to \widehat S$ with $\widehat F|_{\widehat A}=f^p_0$ and $\widehat F|_{\widehat B}=f^{r+k_0p}_0$.

Without loss generality, we assume $\widehat A$ and $\widehat B$ are truncations of slices $S(\theta^-,t,s,\theta^+)$ and $S(t_{N_1},t',s',s_{N_1})$ by equipotential curves respectively for some $t,s,t',s' \in \R/\Z$. Moreover, we may assume $f^r_0(\mathcal R(t_{N_1}))=\mathcal R(\theta^-)$. For any $\eta\in \mathbb R/\mathbb Z$ and $\kappa>0$, let $\xi(\eta, \kappa)$ denote the intersection point of the external ray $\mathcal R(\eta)$ and the equipotential curve of height $\kappa$.  Let $\Phi:\mathbb  H_+\cap \{z\mid 0<\Im z<1\}\to \C\setminus K(f_0)$ be a map defined as $\Phi(z)=\phi^{-1}_{f_0}(e^{2\pi i z})$ where $\mathbb H_+$ is the right half-plane and $\phi_{f_0}$ is the B\"ottcher coordinate of $f_0$. Let $\gamma_-$ and $\gamma_+$ be the  line segment with endpoints $\{\Phi^{-1}(\xi(t, d^{-N_1-p})), \Phi^{-1}(\xi(t',d^{-N_1-k_0p-r}))\}$ and $\{\Phi^{-1}(\xi(s, d^{-N_1-p})), \Phi^{-1}(\xi(s',d^{-N_1-k_0p-r}))\}$ respectively, where $d=d_1+d_2$. Let $\widehat C$ be the closure of  the bounded component of $\C\setminus(\overline A\cup\overline B\cup \Phi(\gamma_-)\cup\Phi(\gamma_+))$.

Let $A=(0,1/3)\times (-1/6, 1/6)$, $B=(2/3, 1)\times (-1/6,1/6)$, $C=[1/3,2/3]\times [-1/6, 1/6]$ and $S=(0,1)\times (-1/2,1/2)$ and define a map
$$F: A\cup B \to S,$$
as follows:
$$F(z)=\left\{\begin{array}{ll}
3z, &\mbox{ if } z\in A;\\
 3(1-z), &\mbox{ if } z\in B.
\end{array}
\right.
$$

\begin{Lemma}\label{lem:qc homeo}
There is a qc homeomorphism
$H:S\to  \widehat S$ such that
\begin{enumerate}
\item[(i)] $H(A)=\widehat A, H(B)=\widehat B, H(C)=\widehat C,$
\item[(ii)] $H\circ F=\widehat F\circ H$ holds on $ \overline {A\cup B}$.
\end{enumerate}
\end{Lemma}
\begin{proof}

We claim that there is qc map $H_0:S\to  \widehat S$ such that (i) holds and (ii) holds on $\partial A\cup \partial B$ (with $H$ replaced by $H_0$). Indeed, one can easiy construct a homeomorphism  with the desired properties by using B\"ottcher coordinate. It can be made global qc by Beurling-Ahlfors extension since $\partial S, \partial A, \partial B, \partial C$ are all quasi-circles.

Now we can construct inductively a sequence $\{H_n\}_{n=0}^\infty$ of qc maps by pull-back which has the following properties:
\begin{itemize}
\item $H_{n+1}=H_n$ on $S \setminus (A\cup B);$
\item $H_{n}\circ F=\widehat F\circ H_{n+1}$ holds on $\overline {A\cup B}$.
\end{itemize}
These maps $H_n$ have the same maximal dilatation as $H_0$, and they eventually stablize for any point in the set $X=\{z\in  S: F^n(z)\not\in \bigcup A\cup B\mbox{ for some }n\}$. Since $F$ is uniformly expanding, the set $X$ is dense in $S$, it follows that $H_n$ converges to a qc map $H$ which satisfies the requirements.

\end{proof}

 Let
$\widehat{\mathcal{N}}=\{z\in \widehat S: \exists n\ge 1\text{ such that }\widehat F^n(z) \text{ is well-defined and belongs to}~~\widehat C\}.$  Note that $H(\mathcal{N})=\widehat{\mathcal{N}}$,  where $\mathcal{N}$ is as in (\ref{eqn:setN}).  The following Lemma follows immediately by Lemma~\ref{lem:qc homeo}, Theorem~\ref{thm:Kahn} and Lemma~\ref{qc1}.
\begin{Lemma} $\mathcal{QD}(\widehat{\mathcal{N}}, \widehat S)<\infty.$
\end{Lemma}

\begin{proof}[Proof of Theorem~\ref{thm:Kqcpuzzle'}]
For each $z\in \partial Y\cap J(f_0)$, there exists a unique geometric ray-pair $ \Gamma(z)\ni z$ such that each componet of $ \Gamma(z) \setminus \{z\}$ intesects $\overline Y$. Let $m(z)$ be the first moment for which $f^{m(z)}_0(\Gamma(z))=\partial S_0=\mathcal R(\theta^-,\theta^+)$. Recall that  $g$ is a single-valued univalent branch  of $f^{-p}_0$ on a neighborhood of $\widehat S(N_1)$ such that $g(\tau_0)=\tau_0$ and $h=h_{k_0}$ is a single-valued univalent branch  of $f^{-r}_0$ on a neighborhood of $\widehat S(N_1+k_0p)$ such that $h(\tau_0)=\tau'_0$.  As $g^k(\widehat S(N_1))$ shrinks to $\tau_0$ as $k$ tends to infinity, there exists $k_1=k_1(z)>k_0$ large enough so that $\overline{g^k(\widehat S(N_1))}$ is disjoint from $\bigcup_{j=0}^{m}f^j_0(\{0,c_0\})$. So we can define a  single-valued univalent branch $\psi$ of $f^{-m(z)}_0$ on a neighborhood of $\overline{g^{k_1}(\widehat S(N_1))}$ such that $\psi(\tau_0)=z$.

Take $N(z)=3(m(z)+N_1+k_1p)$ and let $W=W(z):=\psi\circ g^{k_1}(\widehat S(N_1))$. Then any landing domain of $Y^{(N(z))}$ does not intersect $\partial W$ and $\partial A \cup \partial B$. Therefore, $L_{N(z)}\cap \widehat S \subset \widehat{\mathcal{N}}$, which implies $$\mathcal{QD}(L_{N(z)}\cap \widehat S, \widehat S)\le \mathcal{QD}(\widehat{\mathcal{N}}, \widehat S)=Q<\infty.$$
Note that $f_0^{m(z)+k_1p}$ maps a neighborhood of $W$ conformally onto a neighborhood of $\widehat S$ and $f_0^{m(z)+k_1p}(L_{N(z)}\cap W ) \subset L_{N(z)}\cap \widehat S$.  For otherwise there exists a moment $0\le T<m(z)+k_1p$ such that $f^T_0(W)$ contains a critical point. This contradicts that $f^{m(z)+k_1p}_0|_W$ is a conformal map. Therefore, by  Lemma~\ref{qc1},  
\[\mathcal{QD}(L_{N(z)}\cap W, W) \le \mathcal{QD}(L_{N(z)}\cap \widehat S, \widehat S)<Q.\]

Let $E:=\{z\mid z\in \partial Y\cap J(f_0)\}$. Take $N=\max\{ N(z)\mid z\in E\}$.  Then 
\[\sup\limits_{z\in E} \mathcal{QD}(L_{N}\cap W(z), W(z)) \le \sup\limits_{z\in E}\mathcal{QD}(L_{N(z)}\cap W(z), W(z)<Q.\]
Since $\min \{k_1(z)\mid z\in E\}$ can be taken arbitrarily large, we may assume $W(z)$'s are mutually disjoint.  Note that $L_N \setminus \bigcup_{z\in E} W(z)$ is compactly contained in $Y$. By Lemma~\ref{lem:qc3}, 
\[\mathcal{QD}(L_N\cap Y,Y)<\infty.\]

\end{proof}

\section{Maps in $\mathcal C(\lambda_{f_0})$}
Recall that $Z_0=\mathrm{orb}(\tau_0)$ is an $f_0$-admissible given by Theorem~\ref{thm:puzzle}. In this section, we study the combinatorial properties for the maps in $\mathcal C(\lambda_{f_0})$. More precisely, we show all the maps in $\mathcal C(\lambda_{f_0})$ have the same puzzle structure as $f_0$.  Moreover, for any $f\in \mathcal C(\lambda_{f_0})$  there is a qc weak pseudo-conjugacy between $f_0$ and $f$ up to an arbitrarily given depth.

\subsection{Weak pseudo-conjugacy and pseudo-conjugacy}

We say a polynomial  $\widetilde f $ {\em has the same combinatorics as $f_0$ up to depth $n$} if there exists a  homemorphism $H:\C\to \C$ which satisfies the following:
\begin{itemize}
\item $H$ is identity in the B\"ottcher coordinate for $f_0$ and $\widetilde f$  near $\infty$,i.e., $H=\phi^{-1}_{\widetilde f}\circ \phi_{f_0}$ in a neighborhood of infinity, where $\phi_{f_0}$ and $\phi_{\widetilde f}$ are B\"ottcher map for $f_0$ and $\widetilde f$ respectively;
\item  $H\circ f_0=\widetilde f\circ H$ outside the union of all the $f_0$-puzzle pieces of depth $n$. (Note that all the puzzle pieces are open in our definition.)
 \end{itemize}
 Such a homeomorphism $H$ is called a {\em weak pseudo-conjugacy up to depth $n$} between $f_0$ and $\widetilde f$.
If $H\circ f_0=\widetilde f\circ H$ holds everywhere outside all the critical puzzle pieces of depth $n$, then such an $H$ is called a {\em pseudo-conjugacy up to depth $n$ }.

\begin{Lemma}\label{lem:same combinatorics}For any $f\in \mathcal C(\lambda_{f_0})$, $f$  has the same combinatorics as $f_0$ up to an arbitrary depth.
\end{Lemma}
\begin{proof}
By definition, it suffices to show that for any $k$, there exists a homeomorphism $H_k:\C\to\C$ with the following properties:
\begin{itemize}
\item $H_k$ is identity in the B\"ottcher coordinate for $f_0$ and $f$  near $\infty$, i.e., $H_k=\phi^{-1}_{ f}\circ \phi_{f_0}$ in a neighborhood of infinity, where $\phi_{f_0}$ and $\phi_{ f}$ are B\"ottcher map for $f_0$ and $f$ respectively;
\item $H_k\circ f_0=f\circ H_k$ holds everywhere outside any puzzle piece of depth $k$.
\end{itemize}
Indeed, such a homeomorphism can be constructed as following. First let $H_k=\phi^{-1}_{ f}\circ \phi_{f_0}$ outside the union of puzzle pieces of depth $k$. By Lemma~\ref{lem:nondegenerate}, such an $H_k$ is well defined. Then for each puzzle piece $Y$ of depth $k$, we extend $H_k|_{\partial Y}$ into $Y$ to be a homeomorphism in an arbitrary way. 
\end{proof}

\begin{Lemma}\label{lem:tau repelling}
If an $f_0$-external ray $\mathcal R_{f_0}(t)$ lands at $\tau_0$, then the $f$-external ray $\mathcal R_{f}(t)$ lands at a repelling periodic point of $f$ for any $f\in \mathcal C(\lambda_{f_0})$.
\end{Lemma}

\begin{proof}
To obtain a contradiction, we suppose that the landing point $\tau_f $ of $\mathcal R_{f}(t)$ is parabolic. Then $\tau_f$ must attract at least one of the critical points of $f$. Let $\mathrm{Crit}(f)$ denote the set of the critical points of $f$.
By Proposition~\ref{prop:noncrpiece} and Theorem~\ref{thm:puzzle}, there exists $k_0>0$ such that $\partial Y^{(k-m_0)} \cap \mathrm{orb}(\tau_0)=\emptyset$ and $f^{m_0}_0: Y^{(k)}\to Y^{(k-m_0)}$ is an APL  map for all $k\ge k_0$. For any $k\ge k_0$, let $H_k$ be a weak pseudo-conjugacy between $f_0$ and $f$.  We claim  $\mathrm{Crit}(f) \subset H_k(Y^{(k)}\cup X^{(k)})$.
Indeed, $f_0:Y\to f_0(Y)$ is a conformal map for any puzzle piece $Y\notin \{Y^{(k)},X^{(k)}\}$ of depth $k$. Thus $f_0|_{\partial Y}:\partial Y\to \partial f_0(Y)$ is a homeomorphism, and so is $f:\partial H_k(Y) \to \partial H_k(f_0(Y))=\partial f(H_k(Y))$ since $H_k$ is a conjugacy between $f_0$ and $f$ on $\partial Y$. This implies $f:H_k(Y) \to f(H_k(Y))$ is a conformal map. Hence $H_k(Y)$ contains none of the critical points of $f$.

By a similar argument, $f:H_k(Y^{(k_0)}) \to H_k(f(Y^{(k_0)}))$ is a holomorphic proper map of degree $d_1$. Moreover, the orbit of any $c\in \mathrm{Crit}(f)\cap H_k(Y^{(k_0)})$ lies in $\bigcup_{j=1}^{m_0}f^j(H_k(Y^{(k_0)}))=\bigcup_{j=1}^{m_0}H_k(f^{j}_0(Y^{(k_0)})$. Since $\partial Y^{(k-m_0)} \cap \mathrm{orb}(\tau_0)=\emptyset$ and $\tau_f=H_k(\tau_0)$, one can conclude $\mathrm{orb}(\tau_f)$ cannot attract any critical point of $f$ which lies in $H_k(Y^{(k)})$.

On the other hand, it can be proved that $\bigcap_{k\ge k_0} H_k(X^{(k)})$ is a singleton by the same method as in the proof of Lemma~\ref{lem:X shrinking}. Note that such a singleton is a Julia critical point of $f$, which cannot be attracted by $\mathrm{orb}(\tau_f)$ either. 

As we have showed that none of the critical points of $f$ can be attracted to the orbit of $\tau_f$,  a contradiction.


\end{proof}

\begin{Lemma}\label{lem:weak pseudo}If $f\in\mathcal C(\lambda_{f_0})$, then for each $k\in \mathbb N$, there exists a qc weak pseudo-conjugacy between $f_0$ to $f$ up to depth $k$.
\end{Lemma}

\begin{proof}Let $\Theta$ denote the angles $\theta$ for which $\mathcal R_{f_0}(\theta)$ lands at $\tau_0$. By Lemma~\ref{lem:tau repelling},  for each $\theta\in \Theta$, the landing point of $\mathcal R_{f}(\theta)$ is a repelling periodic point of $f$.

For each $k\in \mathbb N$, we can construct such a qc weak pseudo-conjugacy $H_k$ between $f_0$ and $f$ up to depth $k$ by hand. First we define $H_k$ as identity in the  B\"ottcher coordinate for $f_0$ and  $f$ outside the union of the puzzle pieces of depth $k$. It is well-known that for each puzzle piece $Y$ of depth $k$, $H_k|_{\partial Y}$ can extend to a qc homeomorphism form $\overline Y$ onto $\overline{H_k(Y)}$ since $J_f\cap \partial Y$ are repelling (pre)-periodic points (see  \cite[Lemma~5.3]{KSS}).

\end{proof}

\begin{Lemma}\label{lem: landing on boundary}
Let $n_0>m_0$ be an integer such that $Y^{(n_0)},  f_0(Y^{(n_0)}),\ldots, f_0^{m_0-1}(Y^{(n_0)})$ are disjoint and  $F_0=f^{m_0}_0|_{Y^{(n_0)}}:Y^{(n_0)}\to Y^{(n_0-m_0)}$  is an APL map. Suppose $\widetilde f\in \mathcal P_{d_1+d_2}$ has the same combinatorics as $f_0$ up to depth $n_0$. Assume there exists a qc weak pseudo-conjugacy between $f_0$ and $\widetilde f$ up to depth $n_0$ such that  $$\widetilde F=\widetilde f^{m_0}|_{\Phi(Y^{(n_0)}}: \Phi(Y^{(n_0)}) \to \Phi(Y^{(n_0-m_0)})$$ is an APL map of degree $d_1$ with connected filled Julia set $K(\tilde F)$.
Then  $\mathcal R_{\widetilde f}(s)$ and  $\mathcal R_{\widetilde f}(t)$ land at a common point if $\mathcal R_{f_0}(s)$ and  $\mathcal R_{f_0}(t)$ land at a common point which belongs to $\partial U_0$.
\end{Lemma}

\begin{proof} We define a new qc map $H:\C\setminus K(F) \to \C\to \setminus K(\widetilde F)$ such that $H=\Phi$ on $\C\setminus Y^{(n_0)}$ and $H\circ F=\widetilde F\circ H$ on $Y^{(n_0)} \setminus K(F)$. The construction of $H$ is the following.

Let $U=Y^{(n_0)}$, $V=Y^{(n_0-m_0)}$, $\widetilde U=\Phi (Y^{(n_0)})$ and $\widetilde V=\Phi( Y^{(n_0-m_0)})$.
First choose quasidisks $W,W'$ and $\widetilde W,\widetilde W'$ such that $K(F) \Subset W\Subset W'\Subset U$ and  $K(\widetilde F) \Subset \widetilde W\Subset \widetilde W'\Subset \widetilde U$ respectively. Choose a qc homeomorphism $\Psi:W\setminus K(F) \to \widetilde W\setminus K(\widetilde F)$.  Define $H_0=\Phi$ on $\C \setminus W'$ and  $H_0=\Psi$ on $W\setminus K(F)$. Then  interpolate quasiconformally to obtain a qc homeomorphism  $H_0:\C\setminus K(F) \to \C\setminus K(\widetilde F)$.
Set $\widehat H_0=H_0|_{V\setminus K(F)}$. Since  $F:U\setminus K(F)\to V\setminus K(F)$ and $\widetilde F: \widetilde U \to K(\widetilde F)$
are coverings of the same degree , we lift $\widehat H_0$ to a  $K$-qc map $\widehat H_1:U\setminus K(F)\to\widetilde U\setminus K(\widetilde F)$ such that $\widehat H_1|_{\partial U}=\widehat H_0=H_0=\Phi$. Let $H_1:\C\setminus K(F) \to \C\setminus K(\tilde F)$ be a $K$-qc map which is obtained by glueing $H_0|_{\C\setminus U}$ and $\widehat H_1$ together. By definition,  $\widetilde F\circ H_1=H_0\circ F$ on $U\setminus K(F)$. Inductively, we can obtain a sequence  of $K$-qc maps $H_n:\C\setminus K(F) \to \C\setminus K(\tilde F)$ such that $\widetilde F\circ H_{n+1}=H_n\circ F$ on $U\setminus K(F)$. Since $H_n$ are $K$-qc  and eventually stable on $\C\setminus K(F)$, $H_n$ converges to a $K$-qc map $H:\C\setminus K(F) \to \C\setminus K(\tilde F)$ which is desired.
 
 Note that $H$ maps $\mathcal R_{f_0}(s)$ and $\mathcal R_{f_0}(t)$ onto  $\mathcal R_{\widetilde f}(s)$ and $\mathcal R_{\widetilde f}(t)$ respectively. 
Let $\varphi:\C\setminus\overline{\mathbb D}\to \C\setminus K(\widetilde F)$ denote a Riemann mapping. Since $\partial K(F)=\partial U_0$ is a Jordan curve, 
$\varphi^{-1}\circ H$ extends to a homeomorphism from $\C\setminus K(F)$ onto $\C \setminus \mathbb D$. Let $\zeta(\theta)$ and $\widetilde{\zeta}(\theta)$ denote the landing point of the external ray $\mathcal R_{ f_0}(\theta)$ and $\mathcal R_{\widetilde f}(\theta)$ of angle $\theta \in \mathbb Q/\Z$ respectively.  Note that $\varphi^{-1}(\mathcal R_{\widetilde f}(s))=\varphi^{-1}\circ H(\mathcal R_{f_0}(s))$ lands at $x=\varphi^{-1}\circ H(\zeta(s))=\varphi^{-1}\circ H(\zeta(t))$ and $\mathcal R_{\widetilde f}(s)$ lands at $\widetilde{\zeta}(s)$. By Lindel\"of Theorem (see \cite[Theorem~6.3]{McM1}), the hyperbolic geodesic $\varphi([0,x])$  lands at $\widetilde{\zeta}(s)$. In the same manner we can see that  $\varphi([0,x])$ also lands at $\widetilde{\zeta}(t)$. Thus $\widetilde{\zeta}(s)=\widetilde{\zeta}(t)$, in other words,  $\mathcal R_{\widetilde f}(s)$ and  $\mathcal R_{\widetilde f}(t)$ land at a common point.
\end{proof}

\section{Quasiconformal surgery}\label{sec:qc surgery}
Recall $f_0$ is a primitive  polynomial with a periodic critical point $0$ of period $m_0$ and a Julia critical critical point $c_0$ at which $f_0$ is non-renormalizable. Let $Z_0=\mathrm{orb}(\tau_0)$ be given by Theorem~\ref{thm:puzzle}.   Fix $\kappa_0$ large enough so that the following holds:

 \begin{itemize}
 \item $\overline {Y^{(\kappa_0)}},\overline{f_0(Y^{(\kappa_0)})},\ldots, \overline{f^{m_0-1}_0(Y^{(\kappa_0)})}$ are disjoint;
 \item $f^{m_0}_0:Y^{(\kappa_0)} \to Y^{(\kappa_0-m_0)}$ is an APL map with filled Julia set $\overline{U_0}$.
  \end{itemize}
 The existence of  $\kappa_0$ follows from Theorem~\ref{thm:puzzle} and the primitive property of $f_0$. 
 
The aim of this section is to prove Theorem~B. In~\cite{SW}, qc surgery and  Thurston Algorithm was applied successfully to construct the desired tuning directly. Unfortunately, we cannot construct our desired tuning in our case directly. As the orbit of the free critical point may accumulate on the boundary $\partial U_0$ of the Fatou component $U_0$, it is impossible to respect all the combinatorics when one does qc surgery. Instead,  we first show that for an arbitrary given depth there is a polynomial $P$ which has the same combinatorics as $f_0$ up to this depth and has some other dynamical properties. More precisely, we prove: 
\begin{Theorem}\label{thm:sequence}Let $f_0\in \mathcal C_{d_1+d_2}$ be a primitive  polynomial with a periodic critical point $0$ of multiplicity $d_1-1$ and a Julia critical point $c_0$ of multiplicity $d_2$. Assume $f_0$ has no neutral cycle and is not renormalizable at $c_0$. Fix an internal angle system $\alpha$ for $f_0$ and choose an external angle $\theta_0$ such that $\mathcal R_{f_0}(\theta_0)$ lands at $\alpha^{-1}(0)$. Then for any $k>\kappa_0$ and $g\in \mathcal C_{d_1}$, there exists a polynomial $f_k\in \mathcal P_{d_1+d_2}$ with the following properties:
\begin{enumerate}
\item there exists a qc weak pseudo-conjugacy $\Phi_k$ between $f_0$ and $f_k$ up to depth $k$;
\item $\max \{G_{f_k}(c')\mid  c' ~\text{is a critical point of}~f_k\}\le d^{-k}$, where $d=d_1+d_2$;
\item $f_k^{m_0}:\Phi_k(Y^{(\kappa_0)}) \to \Phi_k(Y^{(\kappa_0-m_0)})$ is an APL map with a polynomial-like restriction which is hybrid equivalent to $g$ and the hybrid conjugacy $\varphi_k$ can be chosen so that $\varphi_k$ sends $\mathcal R_{f_k}(\theta_0)$ to $\mathcal R_g(0)$.
\end{enumerate}
\end{Theorem}
\begin{Remark}In $(3)$ of Theorem~\ref{thm:sequence}, the statement $\varphi_k$ sends $\mathcal R_{f_k}(\theta_0)$ to $\mathcal R_g(0)$ means that there exist neighborhoods $W_k$ and $W$ of the landing points of $\mathcal R_{f_k}(\theta_0)$ and  $\mathcal R_g(0)$ respectively such that $\varphi_k$ maps the connected component of $\overline{\mathcal R_{f_k}(\theta_0)}\cap W_k$ which intersects $J_{f_k}$ onto the connected component of $\overline{\mathcal R_g(0)} \cap W$ which intersects $J_{g}$.
\end{Remark}

The proof is based on  quasiconformal surgery and techniques developed in~ \cite{SW}. Then we take a suitable limit and use some combinatorial facts to show that:

\begin{Theorem B}Under the assumptions in the Theorem~A, $f_0$ can be tuned with any polynomial  $g\in\mathcal C_{d_1}$.
\end{Theorem B}

  \subsection{Constructions of quasiregular maps}Let us recall some terminologies and notations in \cite{SW}, we say a quasiregular map $\widetilde{f}$ is a {\em quasi-polynomial} if $\widetilde{f}^{-1}(\infty)=\infty$ and $\widetilde{f}$ is holomorphic in a neighborhood of $\infty$. An open set $\mathcal U$ is called {\em nice with respect to $\widetilde f$} if each component of $\mathcal U$ is a Jordan domain and $\widetilde{f}^k(\partial \mathcal U)\cap \mathcal U=\emptyset$ for each $k\ge 1$. For  a nice open set $U$, let
$$D(\mathcal U,\widetilde f)=\{z\in \C: \exists n\ge 1\text{ such that } \widetilde{f}^n(z)\in \mathcal U\}$$
denote the domain of the first entry map to $\mathcal U$ under $\widetilde f$.

A nice open set $\mathcal U$ is said to be {\em $\widetilde f$-free} if $\mathcal U \cap P(\widetilde f)=\emptyset$, where $P(\widetilde f)$ is the closure of the orbit of the ramification points of  $P(\widetilde f)$.

Let $\mathcal B$ be a nice open set with respect to $\widetilde f$.  We say   $\mathcal{B}$  is
{\em $M$-nice with respect to $\widetilde f$} if  $\mathcal  B$ is nice with respect to $\widetilde f$ and for each component $B$ of $\mathcal{B}$, the following three conditions hold:
\begin{equation}\label{eqn:shapeB}
\mathrm{diam}(B)^2\le M\mathrm{area}(B) ;
\end{equation}
\begin{equation}\label{eqn:Bretsmall}
\frac{\mathrm{area}(B\setminus D(\mathcal  B,\widetilde f))}{\mathrm{area}(B)}>M^{-1};
\end{equation}
\begin{equation}\label{eqn:Bretqc}
\mathcal{QD}(D(\mathcal B,\widetilde f)\cap B,B)\le M,
\end{equation}
where $\mathcal{QD}$ is as defined in \S\ref{sec:Kahn}.


 \begin{Lemma}\label{lem:surgery} Let $g\in \mathcal C_{d_1}$. Then for any positive integer $k>\kappa_0$, there is a quasi-polynomial  $\widetilde f_k$ of degree $d_1+d_2$ and an open set $\mathcal B_k \Subset Y^{(k)}$  satisfying the following properties:
 \begin{itemize} 
 \item $\widetilde f_k=f_0$ on $\C\setminus Y^{(k)}$;
  \item $\mathcal B_k$ is $\widetilde f_k$-free and $M_k$-nice for some $M_k>0$;
 \item there exists a positive integer $T_k$ such that 
 for every $z$ and $n\ge 1$,  if $\widetilde{f}_k^j(z)\not\in\mathcal{B}_k$ for each $0\le j <n$, then  $\#\{0\le j<n:  \widetilde f^j_k(z) \in \mathcal A_k\}\le T_k$, where $\mathcal A_k=\{z\mid \bar\partial \widetilde f_k(z) \ne 0\}$;
 \item $\widetilde F_k=\widetilde f^{m_0}_k|_{Y^{(\kappa_0)}}:Y^{(\kappa_0)}\to  Y^{(\kappa_0-m_0)}$ is a quasiregular map of degree $d_1$ with a polynomial-like restriction $P_k$ which is hybrid equivalent to $g$.  Furthermore, $K(\widetilde F_k):=\bigcap_{n=0}^{\infty} \widetilde F^{-n}_k(Y^{(\kappa_0)}) $ is  disjoint from $ \bigcup_{n=0}^\infty\widetilde{f}_k^{-n}(\mathcal{B}_k\cup \mathcal A_k)$ and each orbit of $\widetilde F_k$ passes through $\mathcal A_k$ at most $T_k$ times.
 \item there exists a $\widetilde f_k^{m_0}$-invariant ray $\gamma_k$ landing on $K(\widetilde F_k)$ such that $\gamma_k\cap (\C\setminus Y^{(\kappa_0)})=\mathcal R_{f_0}(\theta_0)$ and there exists a hybrid conjugacy $\psi_k$ between $P_k$ and $g$ which sends $\gamma_k$ to $\mathcal R_g(0)$.
 \end{itemize}
  \end{Lemma}
We postpone the proof of this lemma to the end of this subsection and show now how it implies Theorem~\ref{thm:sequence}. 
 \begin{proof}[Proof of Theorem~\ref{thm:sequence}]
 Note that $\widetilde f_k$ satisfies the assumption of~\cite[Theorem~5]{SW}. Thus it follows there exists a polynomial $f_k\in \mathcal P_{d_1+d_2}$ and a continuous surjection $h_k:\C\to \C$ such that $f_k\circ h_k=h_k\circ \widetilde f_k$. Moreover, $h_k$ coincides with some qc map $\lambda_k$ on  $\C\setminus \bigcup_{n=0}^\infty \widetilde{f}_k^{-n}(\mathcal{B}_k)$ and $h_k$ is comformal in a neighborhood of $\infty$. Since $\widetilde f_k=f_0$ outside $Y^{(k)}$ and  $h_k$ is a conformal conjugacy between $f_k$ and $\widetilde f_k$ near infinity, $f_k\circ h_k=h_k\circ f_0$  in a neighborhood of $\infty$. Thus $h_k=\lambda_k$ is identity  in the  B\"ottcher coordinate for $f_0$ and $f_k$ near infinity. This implies the $f_k^{m_0}$-invariant ray $h_k(\gamma_k)$ is the external ray $\mathcal R_{f_k}(\theta_0)$.
 
 First we show that $\lambda_k$ is a qc weak pseudo-conjugacy between $f_0$ and $f_k$ up to depth $k$. Let $\mathcal Y_k$ denote the set of puzzle pieces of depth $k$ and let $\mathcal V_k=\C\setminus \bigcup_{Y\in \mathcal Y_k} Y$. It remains to show that $\lambda_k\circ f_0=f_k\circ \lambda_k$ on $\mathcal V_k$. It suffices to show  $\lambda_k=h_k$ on $\mathcal V_k$, which follows easily from the fact that $\mathcal V_k$ is disjoint from  $\bigcup_{n=0}^\infty \widetilde{f}_k^{-n}(\mathcal{B}_k)$. Indeed,  $\mathcal V_k$ is disjiont from   $\bigcup_{n=0}^\infty \widetilde{f}_k^{-n}(Y^{(k)})$  since $\widetilde f_k=f_0$ outside $Y^{(k)}$. 
 
 Now we turn to prove that $\max \{G_{f_k}(c')\mid c'~\text{is a critical point of}~f_k\}\le 1/d^k$, where $d=d_1+d_2$. Note that $\phi:=\lambda^{-1}_k$ is qc and identity in the  B\"ottcher coordinate for $f_k$ and $f_0$ on $\{z\mid G_{f_k}(z)>1/d^k\}$. Suppose, contrary to our claim, that there is a critical point $c_k$ of $f_k$ with $G_{f_k}(c_k)>1/d^k$.  Then there exists a small disk $D_k$ around $c_k$
 such that $G_{f_k}(z)>1/d^k$ for all $z\in D_k$. Then $\phi\circ f_k=f_0\circ \phi$ on $D_k$ since $D_k$ is disjoint from $\bigcup_{n=0}^\infty \widetilde{f}_k^{-n}(\mathcal{B}_k)$. It follows $\phi(c_k)$ is a critical point of $f_0$ and $G_{f_0}(\phi(c_k))=G_{f_k}(c_k)>1/d^k>0$. Thus $\phi(c_k)$ is in the basin of infinity of $f_0$. This is absurd since $J_{f_0}$ is connected.

 We proceed to show that $F_k:=f^{m_0}_k:h_k(Y^{(\kappa_0)})\to h_k(Y^{(\kappa_0-m_0)})$ is an APL map with a polynomial-like restriction hybrid equivalent to $g$. Since $Y^{(\kappa_0)}$ and $Y^{(\kappa_0-m_0)}$ are nice w.r.t. $\widetilde f_k$ and $\mathcal B_k\Subset Y^{(\kappa_0)}\subset Y^{(\kappa_0-m_0)}$, $\partial Y^{(\kappa_0)}$ and $\partial Y^{(\kappa_0-m_0)}$ lie outside $\bigcup_{n=0}^\infty \widetilde{f}_k^{-n}(\mathcal{B}_k)$. Thus $h_k=\lambda_k$ on $\partial Y^{(\kappa_0)}\cup \partial Y^{(\kappa_0-m_0)}$. Moreover, $h_k=\lambda_k$ and $\bar\partial \lambda_k=0$ a.e. on $K(\widetilde F_k)$ since $K(\widetilde F_k)$ is disjoint from $\bigcup_{n=0}^\infty\widetilde{f}_k^{-n}(\mathcal{B}_k\cup \mathcal A_k)$. Note that $\widetilde{f}_k^{-n}(\mathcal{B}_k)$ is a countable union of disjoint Jordan domains which are nice w.r.t $\widetilde f_k$ and $h_k=\lambda_k$ on $\C\setminus \widetilde{f}_k^{-n}(\mathcal{B}_k)$. It follows $\lambda_k$ is homotopic to $h_k$ rel $K(\widetilde F_k)$. Therefore, by homotopy lifting,  there is a sequence of qc maps $H_j:Y^{(\kappa_0-m_0)} \to \lambda_k(Y^{(\kappa_0-m_0)})$ homotopic to $h_k$ rel $K(\widetilde F_k)$ such that $H_j=\lambda_k$ on $Y^{(\kappa_0-m_0)}\setminus\overline{Y^{(\kappa_0)}}$ and  $F_k\circ H_{j+1}=H_j\circ \widetilde F_k$. Since  $F_k$ is holomorphic and each orbit of $\widetilde F_k$  can pass through $\mathcal A_k$ at most $T_k$ times, the maximal dilatation of $H_j$ is uniformly bounded. As $H_j(z)$  eventually stablizes for all $z\in Y^{(\kappa_0-m_0)}$, $H_j$ converges to a qc map $H$. Clearly, $\bar\partial H=0$  holds a.e. on $K(\widetilde F_k)$ since  so does $H_j$ for all $j$. Assume $P_k=\widetilde F_k|_{\Omega_k}:\Omega_k\to \Omega'_k$ is a polynomial-like restriction of $\widetilde F_k$ which is hybrid equivalent to $g$ and $\psi_k$ is a hybrid conjugacy between $P_k$ and $g$ which sends $\gamma_k$ to $\mathcal R_{g}(0)$ . Then $F_k|_{H(\Omega_k)}:H(\Omega_k) \to H(\Omega'_k)$ is also a polynomial-like map hybrid equivalent to $g$. Note that $H(\gamma_k\cap Y^{(\kappa_0)})=h_k(\gamma_k)\cap H(Y^{(\kappa_0)})$, which is a part of the external ray $\mathcal R_{f_k}(\theta_0)$. Thus $\psi_k\circ H^{-1}$ sends the external ray  $\mathcal R_{f_k}(\theta_0)$ to $\mathcal R_{g}(0)$.
  \end{proof}
 
Now we prove Lemma~\ref{lem:surgery} to finish this subsection.

\begin{proof}[Proof of Lemma~\ref{lem:surgery}] Since $f_0$ is primitive and non-renormalizable at $c_0$, the critical orbit $\mathrm{orb}(c_0)\cap \overline{U_0} =\emptyset$. By Theorem~\ref{thm:puzzle}, $Y^{(n)}$ shrinks to $\overline{U_0}$ as $n$ tends to infinity.
 Thus for any positive integer $k> \kappa_0$, we can choose a positive integer $\ell=\ell(k)>k$ with the following properties:
 \begin{itemize}
 \item $\overline {Y^{(\ell)}},\overline{f_0(Y^{(\ell)})},\ldots, \overline{f^{m_0-1}_0(Y^{(\ell)})}$ are disjoint;
 \item $f^{m_0}_0:Y^{(\ell+m_0)} \to Y^{(\ell)}$ is an APL map with filled Julia set $\overline{U_0}$.
  \end{itemize}
  By Theorem~\ref{thm:Kqcpuzzle'}, there exists a positive integer $N_0=N_0(\ell)>\ell+m_0$ such that 
\[\max \{\mathcal{QD}(L_{N_0}\cap Y^{(\ell)},Y^{(\ell)}), \mathcal{QD}(L_{N_0}\cap X^{(\ell)},X^{(\ell)}\}<\infty.\]
It follows from Theorem~\ref{thm:puzzle} that there exists a positive integer $i_0$ such that $Y^{(N_0+i_0m_0)} \Subset Y^{(N_0)}$. 
Let $A$  denote the annuli $Y^{(N_0)} \setminus \overline{Y^{(N_0+i_0m_0)}}$.
%
%
%
%

{\bf Case 1.}There exists a first moment $s_0$ such that $x_0:=f^{s_0}_0(c_0) \in A$. 

We define a quasiregular map $\widetilde f_0:A \to f_0(A)$ such that $d^{-N_0}<G_{f_0}(\widetilde f_0(x_0))<{d^{-N_0+1}}$, $d=d_1+d_2$ and $\widetilde f_0=f_0$ on $\mathcal R_{f_0}(\theta_0)\cap A$ as following. Fix $y_0\in f_0(A')$ with $G_{f_0}(y_0)>d^{-N_0}$. Define $\widetilde f_0=f_0$ on $\partial A\cup( \mathcal R_{f_0}(\theta_0)\cap A)$ and interpolate quasiregularly on $A$ so that $\widetilde f_0(x_0)=y_0$. 

 Now choose $r_0$ large enough so that $\{c_0, f_0(c_0),\ldots, f^{s_0}_0(c_0)\} \cap \overline{Y^{(r_0)}}=\emptyset$. Note that $f^{m_0}_0:Y^{(r_0+m_0)} \to Y^{(r_0)}$ is an APL map with filled Julia set $\overline{U_0}$. Since $\overline{U_0} \Subset  Y^{(r_0+m_0)}$,  by \cite[Lemma~2.4]{LY}, there exist quasidisks $U \Subset V \subset Y^{(r_0+m_0)}$ such that $f^{m_0}_0:U\to V$ is a polynomial-like map with filled Julia set $\overline{U_0}$. Then take a polynomial-like restriction $g|_{\widehat W}:\widehat W\to \widehat U$ of $g$. The quasidisk $U$  can be chosen so that $\mathcal R_{f_0}(\theta_0)$  intersects $\partial U$  transversally (at a unique point). Let $\xi$  denote the corresponding intersecting point. Similarly, we can assume $\mathcal R_g(0)$ intersects $\partial \widehat U$ at a unique point $\zeta$. Let $\xi'\in \mathcal R_{f_0}(\theta_0)$ and $\zeta'\in \mathcal R_g(0)$ so that $f^{m_0}_0(\xi')=\xi$ and $g(\zeta')=\zeta$.
  Let $\varphi:U\to \widehat U$ be the unique conformal map from $U$ onto $\widehat U$ with $\varphi(\xi')=\zeta'$ and $\varphi(\xi)=\zeta$. Set $W=\varphi^{-1}(\widehat W)$.

Let us define $\widehat F_k= \varphi^{-1}\circ g\circ \varphi$ on $\overline{W}$ and $\widehat F_k=f^{m_0}_0$ on $\partial U$. Then we interpolate quasiregularly on $U\setminus \overline{W}$ so that $\widehat F_k(\varphi^{-1}(\mathcal R_g(0)\cap(\widehat U\setminus \widehat W)))=f^{m_0}_0(\mathcal R_{f_0}(\theta_0))$.  Let $\mathcal I$ denote the inverse map of the first entry map $f^{m_0-1}_0:f_0(U)\to V$ which is conformal. Then we define $\widetilde f_k$ as following:
\[\widetilde f_k(z)=
\begin{cases}
\widetilde f_0(z), & z\in A\\
\mathcal I \circ \widehat F_k(z), & z\in U\\
f_0(z), & z\in \C \setminus (A\cup U). 
\end{cases}\]
Note that $\widetilde f_k(\varphi^{-1}(\mathcal R_g(0)\cap(\widehat U\setminus \widehat W)))=f_0(\mathcal R_{f_0}(\theta_0))$.

In this case, $\bar\partial \widetilde f_k=0$ holds almost everywhere outside  $\mathcal A_k={A} \cup U\setminus \overline{W}$. Let $R(\mathcal A_k)$ denote the return domian to $\mathcal A_k$ under $\widetilde f_k$.  For any $z\in \mathcal A_k$, there exists a smallest positive integer $\iota(z)$ such that $\widetilde f_k^{\iota(z)}(z) \in Y^{(\ell)}\setminus \overline{Y^{(\ell+m_0)}}$. It is easy to see $T:=\sup_{z\in \mathcal A_k} \iota(z) <\infty$.
 Set
\[\Omega:=\{z\in Y^{(\ell)}\setminus \overline{Y^{(\ell+m_0)}}\mid \exists j>1~\text{such that}~\widetilde f_k^{j}(z) \in Y^{(\ell)}\cup X^{(\ell)}\}.\]
Let ${\mathcal B}_k$ be the union of the component of  $D(\Omega)$ which intersects $R(\mathcal A_k)$, where $D(\Omega)$ is the domian of the  first entry map to $\Omega$ under $\widetilde f_k$.

\medskip
{\bf Fact 1.} {\em $\mathcal B_k$ is $\widetilde f_k$-free}.

Clearly, $\mathcal B_k \subset Y^{(N_0)}$ since $A\subset Y^{(N_0)}$ and $\partial Y^{(N_0)} $ is nice with respect to $\widetilde f$. Note that $\widetilde f_k^{j}(c_0)=f^{j-s_0-1}_0\circ \widetilde f_0(x_0)$ for all $j>s_0$ by the definition of $\widetilde f_k$.
Thus $G_{f_0}(\widetilde f_k^{j}(c_0)) >d^{-N_0}$ for all $j>s_0$, which  implies $\widetilde f_k^{j}(c_0)\notin \overline{Y^{(N_0)}}$ for all $j>s_0$. So $Y^{(N_0)} \cap \{\widetilde f^j(c_0)\}_{ j\in \mathbb N }=\{x_0\}$ by the definition of $s_0$ and $\widetilde f_k$. 
It suffices to show $x_0 \notin \overline{\mathcal B_k}$. If $x_0\in \overline{\mathcal B_k}$, then there exists a sequence $\{z_j\}\subset \mathcal B_k$ such that $z_j\to x_0$.  So $\widetilde f_k(z_j) \to f_0(y_0)$. Since $d^{-N_0}<G_{f_0}(y_0)<d^{N_0+1}$, $d^{-N_0}<G_{f_0}(\widetilde f_k(z_j))<d^{-N_0+1}$ for all $j$ large. Thus the orbit of  $\widetilde f_k(z_j)$ under $\widetilde f_k$ can never pass  through $Y^{\ell}$ after it enters into $Y^{(\ell)}\setminus \overline{Y^{(\ell+m_0)}}$.  This implies $\widetilde f_k(z_j) \notin D(\Omega)$  and so does $z_j$ for $j$ large, a contradiction.

\medskip

{\bf Fact 2.} {\em $\mathcal B_k$ is $M_k$-nice with respect to $\widetilde f_k$ for some $M_k>0$}.

We first observe that $\Omega$ and hence $\mathcal B_k$ is nice with respect to $\widetilde f_k$. Fix a component $B$ of $\mathcal{B}_k$, let $s'$ be the first entry time of $B$ into $\Omega$, and let $t$ denote the return time of $\widetilde{f}_k^{s'}(B)$ into $Y^{(\ell)}\cup X^{(\ell)}$. Then $\widetilde{f}_k^{s'+t}$ maps $B$ homeomorphically onto a component of $Y^{(\ell)}\cup X^{(\ell)}$
and the map $\widetilde{f}_k^{s'+t}|_B$ is $K^T$-qc, where $K$ is the maximal dilatation of $\widetilde f_k$. Indeed, $s'\le T$ since $B$ intersect $\mathcal A_k$. 
Note that 
$\widetilde{f}_k^{s'+t} (D(\mathcal{B}_k)\cap B)\subset L_{N_0}$ since $\mathcal B_k\subset  Y^{(N_0)}$. It follows from Lemma~\ref{qc1} and Theorem~\ref{thm:Kqcpuzzle'}
that
\[\mathcal{QD}(D(\mathcal {B})\cap B, B) \le (K^T)^{2}\max \{\mathcal{QD}(L_{N_0}\cap Y^{(\ell)},Y^{(\ell)}), \mathcal{QD}(L_{N_0}\cap X^{(\ell)}, X^{(\ell)})\}=:M'_k<\infty.\]
Note that there is $M>0$ such that \[M^2\mathrm{area}(Y^{(\ell)}\setminus L_{N_0})\ge M \mathrm{area}(Y^{(\ell)})\ge \mathrm{diam}(Y^{(\ell)})^2\]
and 
 \[M^2\mathrm{area}(X^{(\ell)}\setminus L_{N_0})\ge M \mathrm{area}(X^{(\ell)})\ge \mathrm{diam}(X^{(\ell)})^2\]
since $Y^{(\ell)}$ and $X^{(\ell)}$ are quasidisks and $(Y^{(\ell)}\setminus L_{N_0}) \cap D_{f_0}(\infty)\ne \emptyset$, $(X^{(\ell)}\setminus L_{N_0}) \cap D_{f_0}(\infty)\ne \emptyset$.
As $\widetilde f^{s'+t}_k|_B$ has a $K^{T}$-qc extension from $B$ onto a neighborhor of $\widetilde f^{s'+t}_k(B)$, there is $M'>0$ (only depends on $M$ and $K^T$) such that
\[M'^2\mathrm{area}(B)\setminus D(\mathcal B_k))\ge M' \mathrm{area}(B)\ge \mathrm{diam}(B)^2\] by area distortion of $K^T$-qc maps.
Indeed, $\partial Y^{(\ell)} $ and $\partial X^{(\ell)}$ does not intersect the orbit of the ramification points of $\widetilde f_k$. Thus there are definite neighborhoods $V(0)$ and $V(c_0)$ of  $\partial Y^{(\ell)} $ and $\partial X^{(\ell)}$ respectively so that $P(\widetilde f_k)\cap (V(0) \setminus Y^{(\ell)} )=\emptyset$ and   $ P(\widetilde f_k)\cap (V(c_0) \setminus X^{(\ell)}) =\emptyset$. Hence $\widetilde f^{s'+t}_k|_B$ has a $K^{T}$-qc extension from $B$ onto a definite neighborhood of $Y^{(\ell)}$ or $X^{(\ell)}$. Let $M_k=\max \{M'_k,M'\}$. Then $\mathcal B_k$ is $M_k$-nice.

\medskip

{\bf Fact 3.}  {\em For every $z$ and $n\ge 1$,  if $\widetilde{f}_k^j(z)\not\in\mathcal{B}_k$ for each $0\le j <n$, then  $\#\{0\le j<n:  \widetilde f^j_k(z) \in  \mathcal A_k\}\le 1$.}

By the definition of $\mathcal B_k$, $R(\mathcal A_k) \subset \mathcal B_k$. So if $\widetilde{f}_k^j(z)\not\in\mathcal{B}_k$ for each $0\le j <n$, then $\widetilde{f}_k^j(z)\not\in R(\mathcal A_k)$ for each $0\le j <n$. This implies $\widetilde{f}_k^j(z)\not\in \mathcal A_k$ for each $1\le j <n$.

\medskip

{\bf Fact 4.} {\em  $\widetilde F_k:=\widetilde f^{m_0}_k:Y^{(\kappa_0)} \to Y^{(\kappa_0-m_0)} $ is a quasiregular map of degree $d_1$ with a polynomial-like restriction which is hybrid equivalent to $g$.  Moreover, $K(\widetilde F_k):=\bigcap_{n=0}^{\infty} \widetilde F^{-n}_k(W_k) $ is  disjoint from $ \bigcup_{n=0}^\infty\widetilde{f}_k^{-n}(\mathcal{B}_k\cup \mathcal A_k)$ and each orbit of $\widetilde F_k$ passes through $\mathcal A_k$ at most $T+1$ times.}

Since $\widetilde f_k: Y^{(\kappa_0)}  \to f_0(Y^{(\kappa_0)}) $ is a qusiregular map of degree $d_1$ and $f^{m_0-1}_0: f_0(Y^{(\kappa_0)})  \to Y^{(\kappa_0-m_0)} $ is a univalent map,  $\widetilde f^{m_0}_k:Y^{(\kappa_0)}  \to Y^{(\kappa_0-m_0)} $ is a quasiregular map of degree $d_1$. Note that $\widetilde f^{m_0}_k|_W=\widehat F_k|_W=\varphi^{-1}\circ g\circ \varphi|_W$ is conformally conjugate to $g$ and $K(\widetilde F_k)=K(\widehat F_k|_W)=\varphi^{-1}(K(g))$. Clearly, $\varphi^{-1}(K(g))$  is  disjoint from $ \bigcup_{n=0}^\infty\widetilde{f}_k^{-n}(\mathcal{B}_k\cup \mathcal A_k)$. For each point $z\in Y^{(\kappa_0)}  \setminus K(\widetilde F_k)$, its orbit will first pass through  $Y^{(\ell)}\setminus \overline{Y^{(\ell+m_0)}}$ and eventually pass through $Y^{(\kappa_0-m_0)}\setminus\overline{Y^{(\kappa_0)}}$ under $\widetilde F_k$. Let $t'$ be the first moment so that $\widetilde F_k^{t'}(z) \in \mathcal A_k$. Then the orbit of $\widetilde F_k^{t'}(z)$ under $\widetilde F_k$ can pass through $\mathcal A_k$ at most $\iota(F_k^{t'}(z)) \le T$ times.

Let $\gamma_k$ denote the union of $\mathcal R_{f_0}(\theta_0)\setminus U$ and $\varphi^{-1}(\mathcal R_g(0))\cap U$.
Clearly, $\varphi$ is a hybrid conjugacy between $\widehat F_k|_W$ to $g$ which sends $\gamma_k$ to $\mathcal R_g(0)$.
\medskip

{\bf Case 2.} The critical orbit $\mathrm{orb}(c_0)$ does not pass through the annulus $A$. Then there exists $\ell'$ large such that $Y^{(\ell')}\cap \overline{\{c_0,f_0(c_0),\ldots\}} =\emptyset$. 
 By Theorem~\ref{thm:Kqcpuzzle'}, there exists a positive integer $N'_0=N'_0(\ell')>\ell'+m_0$ such that $\mathcal{QD}(L_{N'_0}\cap Y^{(\ell')},Y^{(\ell')})<\infty$. 

 Now choose $r_0>N'_0$. Without loss of generality, we assume $r_0$ is the same as in Case~1.  Let $W,U,\widehat W, \widehat U $, $\widehat F_k$ and $\mathcal I$ be as in Case~1. This time, we define 
 \[\widetilde f_k(z)=
\begin{cases}

\mathcal I\circ \widehat F_k(z), & z\in U\\
f_0(z), & z\in \C \setminus  U. 
\end{cases}\]
In this case, $\bar\partial \widetilde f_k=0$ holds almost everywhere outside  $\mathcal A_k=U\setminus \overline{W}$. Again let $R(\mathcal A_k)$ denote the return domian to $\mathcal A_k$ under $\widetilde f_k$.  
Let $$\Omega':=\{z\in Y^{(\ell')}\setminus \overline{Y^{(\ell'+m_0)}}\mid \exists j>1~\text{such that}~\widetilde f_k^{j}(z) \in Y^{(\ell')}\} $$ and let $\mathcal B_k$ be the union of the component of  $D(\Omega')$ which intersects $R(\mathcal A_k)$, where $D(\Omega')$ is the domian of the  first entry map to $\Omega'$ under $\widetilde f_k$. Similar to Case~1, one can check $\widetilde f_k$ is a desired map in Lemma~\ref{lem:surgery}.
 \end{proof}

 \subsection{Proof of Theorem~B}
 Fix $g\in \mathcal C_{d_1}$.
 It follows from Theorem~\ref{thm:sequence} that there exists a sequence $\{f_k\}$ of polynomials and a sequence of $\{\Phi_k\}$ of qc maps satisfies the following:
 \begin{enumerate}
 \item $\Phi_k$ is a qc weak pseudo-conjugacy between $f_0$ and $f_k$ up to depth $k$;
 \item $\max \{G_{f_k}(c')\mid c'~\text{is a critical point of}~f_k\}\le d^{-k}$;
 \item  $f_k^{m_0}:\Phi_k(Y^{(\kappa_0)}) \to \Phi_k(Y^{(\kappa_0-m_0)})$ is an APL map with a polynomial-like restriction which is hybrid equivalent to $g$.
 \end{enumerate}
By~\cite[Proposition~3.6]{BH}, $\{f_k\}$ lies in a compact set. Thus $\{f_k\}$ has a convergent subsequence. Without loss of generality, we assume $f_k\to f$ as $k\to \infty$. We now finish the proof of Theorem~B by claiming that $f$ is our desired tuning.

\begin{proof}[Proof of Theorem~B]
Note that $G_{f}(c)=0$ for any critical point $c$ of $f$ since $$\max \{G_{f_k}(c')\mid c'~\text{is a critical point of}~f_k\}\le d^{-k}$$ holds for all $k\in \mathbb N$.  Thus, $f\in \mathcal C_{d_1+d_2}$. Let $\mathcal E=\{E_j\}_{j\in \mathbb N}$ be the set of all the  equivalence classes of  $\lambda_{f_0}$. We say $E\in \mathcal E$ is admissible if  $\mathcal R_f(t)$ does not land at the grand orbit of a parabolic periodic point or a critical point of $f_0$ or $f$ for all $t\in E$. Let $\mathcal E':=\{E\in\mathcal E\mid E~\text{is admissible}\}$. We claim that for each $E\in \mathcal E'$, $E$ is contained in an equivalence class of $\lambda_{f}$.

Assume $t,s \in E\in \mathcal E'$, we  prove that $\mathcal R_f(t)$ and $\mathcal R_f(s)$  land at a common point. Let $\zeta$ and $\xi$ be the landing point of  $\mathcal R_f(t)$ and $\mathcal R_f(s)$ respectively. By the definition of $\mathcal E'$, $\zeta$ and $\xi$ are repelling periodic point of $f$. By holomorphic motion, we know there exists $\zeta_k\to \zeta$ and $\xi_k\to \xi$ so that $\mathcal R_{f_k}(t)$ and  $\mathcal R_{f_k}(s)$ land at $\zeta_k$ and $\xi_k$ respectively for all $k$ larege (see also ~\cite[Lemma~B.1]{GJ}).

{\bf Case 1.} $\mathcal R_{f_0}(t)$ does not land on $\bigcup\limits_{n=0}^{\infty} f^{-n}_0(\partial U_0)$. Recall that $U_0$ is the Fatou component of $f_0$ which contains $0$. In this case, there exists $k$ large such that $f^n_0(\mathcal R_{f_0}(t) \cup \mathcal R_{f_0}(s))$ lies outside $Y^{(k)}\cup X^{(k)}$ for all $n\in \mathbb N$. It is well known that  we can adjust the weak pseudo-conjugacy $\Phi_k$ to be a pseudo conjugacy.  See \cite[Lemma~4.3]{AKLS} and \cite[section~5]{KSS} for examples, see also Lemma~\ref{lem:pseudo}. Thus there is no loss of generality in assuming $\Phi_k$ is pseudo-conjugacy between to $f_0$ and $f_k$  up to depth $k$.  This implies that $\mathcal R_f(t)=\Phi^{-1}_k(\mathcal R_{f_0}(t))$ and $\mathcal R_f(s)=\Phi^{-1}_k(\mathcal R_{f_0}(s))$ land at a common point for all $k$ large.  So $\zeta_k=\xi_k$ for $k$ large. Hence $\zeta=\xi$.

{\bf Case 2.} $\mathcal R_{f_0}(t)$  lands on $\partial U_0$. It follows from Lemma~\ref{lem: landing on boundary} that $\zeta_k=\xi_k$ for all $k$, and so $\zeta=\xi$.

{\bf Case 3.} $\mathcal R_{f_0}(t)$  lands on $f^{-k_0}_{0}(\partial U_0)$ for some $k_0\ge 1$. Without loss of generality, we assume $k_0$ is the first moment so that $f^{k_0}_0(\mathcal R_{f_0}(t))$ lands on $\partial U_0$.  Let $k_1$ be the smallest positive integer such $f^j_0(\mathcal R_{f_0}(t)\cup \mathcal R_{f_0}(s))$ lies outside $Y^{(k_1)}\cup X^{(k_1)}$ for $0\le j<k_0-1$. Such an integer $k_1$ exists since $Y^{(n)}$ shrinks to $\overline{U_0}$ and $X^{(n)}$ shrinks to $\{c_0\}$. Let $X$ be the puzzle piece of depth $k_0+k_1$ which contains the landing point of $\mathcal R_{f_0}(t)$. Then $f^{k_0}_0:X\to Y^{(k_1)}$ is a conformal map. Note that $f^{k_0}_k:\Phi^{-1}_k(X) \to \Phi^{-1}_k(Y^{(k_1)})$ is a holomorphic proper map for all $k$ large since $\Phi_k$ is a weak pseudo conjugacy between $f_0$ and $f_k$ up to depth $k$. Since $\Phi_k\circ f^{k_0}_m=f^{k_0}_k\circ \Phi_k$ on $\partial X$, $f^{k_0}_k|_{\partial \Phi^{-1}_k(X) }$ is a homeomorphism for $k$ large. Thus $f^{k_0}_k:\Phi^{-1}_k(X) \to \Phi^{-1}_k(Y^{(k_0)})$ is conformal. Moreover, $f^{k_0}_k|_{\Phi^{-1}_k(X)}$ maps the landing points of $\mathcal R_{f_k}(t)$ and $\mathcal R_{f_k}(s)$ to a common point which is a  common landing point of $f^{k_0}_k(\mathcal R_{f_k}(t))$ and $f^{k_0}_k(\mathcal R_{f_k}(s))$ (see Case~2). This implies  $\mathcal R_{f_k}(t)$ and $\mathcal R_{f_k}(s)$ land at a common point for all $k$ large, and so do  $\mathcal R_{f}(t)$ and $\mathcal R_{f}(s)$, since $f^{k_0}_k|_{\Phi^{-1}_k(X)}$ is conformal for $k$ large.

By~\cite[Lemma~8.3]{IK},  $\lambda_{f_0}$ is the smallest equivalence relation in $\mathbb Q/\Z$ contains $(\bigcup_{E\in \mathcal E'} E)\times (\bigcup_{E\in \mathcal E'} E)$. Thus  $\lambda_{f_0} \subset \lambda_{f}$.

It remains to show there exists quasidisks $U\Subset V$ and a qc map $\varphi:\C\to \C$ such that $F:=f^{m_0}|_U:U\to V$ is a polynomial-like map with filled Julia set $K(F)=K(U_0,f)$ which is hybrid equivalent to $g$ and $\varphi$ is a hybrid conjugacy between $F$ and $g$ respecting the external markings.

Let $\Theta$ denote the angles $\eta$ for which $\mathcal R_{f_0}(\eta)$ lands at $\tau_0$. By Lemma~\ref{lem:tau repelling},  for each $\eta\in \Theta$, the landing point of $\mathcal R_{f}(\eta)$ is a repelling periodic point of $f$. By Lemma~\ref{lem:weak pseudo}, there is a qc weak pseudo-conjugacy $\Psi$ between $f_0$ and $f$ up to depth $\kappa_0$. Note that $f^{m_0}_0:Y^{(\kappa_0)}\to Y^{(\kappa_0-m_0)}$ is holomorphic proper map of degree $d_1$, so is $F:=f^{m_0}:\Psi(Y^{(\kappa_0)})\to \Psi(Y^{(\kappa_0-m_0)})$. Since $J_f\cap \partial \Psi(Y^{(\kappa_0)})$ are repelling (pre-)periodic points of $f$, by thickening technique~\cite{Mil4}, there are quasidisks $\Psi(Y^{(\kappa_0)})\subset U,  \Psi(Y^{(\kappa_0-m_0)})\subset V$ such that $f^{m_0}:U\to V$ is a polynomial-like map of  degree $d_1$. Similarly, there are quasidisks $\Phi_k(Y^{(\kappa_0)})\subset U_k,  \Phi_k(Y^{(\kappa_0-m_0)})\subset V_k$ such that  $F_k:=f^{m_0}_k:U_k\to V_k$ is a polynomial-like map of degree $d_1$.  Moreover, we can choose $U_k$, $V_k$ carefully so that $\mathrm{mod}(V_k\setminus \overline{U_k})\ge  \mathrm{mod}(V\setminus\overline{U})/2$ for $k$ large since $f_k\to f$.  By Theorem~\ref{thm:sequence}, there exists qc maps $\psi_k:\C\to \C$ such that $\psi_k$ is a hybrid conjugacy between $F_k$ and $g$ and $\psi_k$ sends the homotopy class of $\mathcal R_{f_k}(\theta_0)$ to the homotopy class of $\mathcal R_g(0)$. Note that we can assume $\psi_k$ are uniformly $K$-qc for some $K\ge 1$ since $\mathrm{mod}(V_k\setminus\overline{U_k}) \ge  \mathrm{mod}(V\setminus\overline{U})/2$ for all $k\in \mathbb N$. By the compactness of normalized $K$-qc maps, without loss of generality, we may assume $\psi_k$ converges uniformly to some $K$-qc map $\varphi$.
Note that $F_k\circ\psi^{-1}_k=\psi^{-1}_k\circ g$ near $K(g)$ and $\bar\partial\psi^{-1}_k=0$ a.e. on $K(g)$. Thus  $F\circ \varphi^{-1}=\varphi^{-1}\circ g$ near $K(g)$ and $\bar\partial\varphi^{-1}=0$ a.e. on $K(g)$. This implies $\varphi$ is a hybrid conjugacy between $F$ and $g$.

We now proceed to show that $\varphi$ sends the homotopy class of $\mathcal R_{f}(\theta)$ to the homotopy class of $\mathcal R_g(0)$. One can easily check that $\overline{\mathcal R_{f_k}(\theta_0)}$ converges uniformly to $\overline{\mathcal R_f(\theta_0)}$ in the sense of Hausdorff topology by using the fact that $\psi_k$ converges uniformly to $\varphi$ and the local dynamical property near the landing point of $\mathcal R_g(0)$. Thus $\varphi^{-1}\circ\psi_k(\mathcal R_{f_k}(\theta_0))$ is homotopic to $\mathcal R_{f}(\theta_0)$ in the sense of paths to $K(F)$ for $k$ sufficiently large.  On the other hand, $\varphi^{-1}\circ\psi_k(\mathcal R_{f_k}(\theta_0))$ is homotopic to $\varphi^{-1}(\mathcal R_g(0))$ since $\psi_k(\mathcal R_{f_k}(\theta_0))$ is homotopic to $\mathcal R_g(0)$ by the definition of $\psi_k$.
Thus $\mathcal R_f(\theta_0)$ is homotopic to $\varphi^{-1}(\mathcal R_g(0))$.

The proof is complete by showing $K(F)=K(U_0,f)$. This is easy by using the weak pesudo-conjugacies between $f_0$ and $f$. The details are left to the readers.

\end{proof}

\begin{Remark}There is another way to prove Theorem~B (suggested by Hiroyuki Inou). One can first find a sequence $\{F_k\}\subset \mathcal C_{d_1+d_2}$ of post-critically finite polynomials such that $F_k$ has the same combinatorics as $f_0$ up to depth $k$ for all $k\in \mathbb N$. See \cite[Section~7]{Kiwi} for a combinatorial argument and this can also been done by qc surgery. Then we use qc surgery to construct quasiregular maps $\widetilde F_k$ satisfies all the conditions  in Lemma~\ref{lem:surgery} with $f_0$ and $\widetilde f_k$ replaced by $F_k$ and $\widetilde F_k$ respectively. The construction is essentially the same as that in Case~2 in the proof of Lemma~\ref{lem:surgery}. By~\cite[Theorem~5.1]{SW}, there exists $f_k\in \mathcal C_{d_1+d_2}$ such that $f_k$ has the same combinatorics as $f_0$ up to depth $k$ and $f_k^{m_0}$ has a polynomial-like restriction which is hybrid equivalent to $g\in \mathcal C_{d_1}$. Again, the limit $f$ of $f_k$ is our desired tuning. We mention here that  using this method one can prove Theorem~B more easily when $d_1=2$. Indeed, $\mathcal C(\lambda_{F_k})$ is an analytic family for all $k\in \mathbb N$. So one may apply Douady-Hubbard's Theorem to obtain $f_k$ when $d_1=2$.
\end{Remark}

\begin{Remark}To fix the idea and not to make the notations too complicate, we only deal with the case that $f_0$ has one periodic critical point and one Julia critical point. Actually the method of the proof of Theorem~B also works for the following case. Let $f\in \mathcal C_{d}$ be a polynomial with only periodic critical points and non-renormalizable Julia critical points. Following~\cite{IK}, one can define the reduced mapping scheme $T(f)$ for $f$. It can be proved by the same method as in proof of Theorem~B that  $f$ can be tuned with any generalized polynomial $g\in \mathcal C(T(f))$. The proof is left to the readers who are interested in it.
\end{Remark}

\section{Bijectivity for the Straightening map}\label{sec:homeomorphism}
Let $f_0$ and $\alpha$ be given as in the assumption of Theorem~A.

For any $f\in \mathcal C(\lambda_{f_0})$, by Lemma~\ref{lem:weak pseudo}, there is a qc weak pseudo-conjugacy $\Psi_f$ between $f$ and $f_0$ up to depth $\kappa_0$. Note that $f^{m_0}: Y^{(\kappa_0)}\to Y^{(\kappa_0-m_0)}$ is an APL map of degree $d_1$,  and so is $F:=f^{m_0}:\Psi_f(Y^{(\kappa_0)})\to \Psi_f(Y^{(\kappa_0-m_0)})$. By \cite[Lemma~2.4]{LY},  $F$ has a polynomial-like restriction with connected filled Julia set $K(F)$. By the discussions in section~2, there is a unique external marking $\Gamma$ for $F$ induced by the internal angle system $\alpha$. It follows Douady-Hubbard Straightening Theorem that  there is a unique polynomial $\chi(f)\in \mathcal C_{d_1}$ such that the polynomial-like restriction of $F$ is hybrid equivalent to $\chi(f)$ if the external markings are respected (the conjugacy sends $\Gamma$ to the standard marking for $\chi(f)$). Therefore, we have a well defined map:
\[\chi:\mathcal C(\lambda_{f_0}) \to \mathcal M, f\mapsto \chi(f).\]
We call $\chi$  {\em the straightening map induced by $f_0$}. Clearly, $f$ is a tuning of $f_0$ by $\chi(f)$. Theorem~B  implies that:
\begin{Theorem}\label{thm:surjection}The staightening map $\chi$ is a surjection.  
\end{Theorem}

%

Now we use the arguments developed in~\cite{AKLS} to show the injectivity of $\chi$:

\begin{Theorem}\label{thm:injection} $\chi$ is an injection.
\end{Theorem}
We postpone the proof to the end of this section and first prove some useful lemmas.
\begin{Lemma}\label{lem:Fatou}Let $f ,\widetilde f\in \mathcal C(\lambda_{f_0})$. For any $k\ge \kappa_0$, let $H_k$ $(\text{resp.}~\widetilde H_k)$ be a qc weak pseudo-conjugacy between $f_0$ and $f$ $(\text{resp.}~\widetilde f)$ up to depth $k$ and  let $\Lambda_k=\widetilde H_k\circ H^{-1}_k$.  If $\chi(f)=\chi(\widetilde f)$, then there exists $K=K(f)>0$ such that  $\Lambda_k|_{\partial{H_k(Y^{(k)})}}$ has a $K$-qc extension $\widehat \Lambda_k$ from $H_k(Y^{(k)})$ onto $\widetilde H_k(Y^{(k)})$ for all $k\ge\kappa_0+m_0$. 
\end{Lemma}
\begin{proof} It is easy to see there exists $K>0$ such that  $\Lambda_{\kappa_0+j}|_{\partial{H_{\kappa_0+j}(Y^{(\kappa_0+j)})}}$  has a desired $K$-qc extension $\widehat \Lambda_{\kappa_0+j}$ for $j=1,\dots,m_0$. Indeed, one can choose  polynomial-like restrictions $F=f^{m_0}:U\to V$ and $\widetilde F=\widetilde f^{m_0}:\widetilde U\to \widetilde V$ such that $U\Subset H_{(\kappa_0+m_0)}(Y^{(\kappa_0+m_0)})$ and $\widetilde U\Subset \widetilde H_{\kappa_0+m_0}(Y^{(\kappa_0+m_0)})$. Let $\psi:U\to \widetilde U$ be a hybrid conjugacy between $f$ and $\widetilde f$ respecting the canonical external marking. Define $\widehat \Lambda_{\kappa_0+j}=\psi$ on $U$ and $\widehat \Lambda_{\kappa_0+j}=\Lambda_{\kappa_0+j}$ on $\partial \Lambda_{\kappa_0+j}(Y^{(\kappa_0+j)})$. Then interpolate quasiconformally. Let $K$ be the maxium of the maximal dilatation of $\widehat H_{\kappa_0+j}$.

Now the Lemma follows easily by induction. Note that $f^{k-\kappa_0-j}:H_{k}(Y^{(k)}) \setminus K(F)\to H_{k}(Y^{\kappa_0+j})\setminus K(F)$ and  $\widetilde f^{k-\kappa_0-j}:\widetilde H_k(Y^{(k)})\setminus K(\widetilde F) \to \widetilde H_k(Y^{\kappa_0+j})\setminus K(\widetilde F)$ are coverings of the same degree for all such $k$ that $m_0\mid k-\kappa_0-j$. We now define  $\widehat \Lambda_k|_{H_{k}(Y^{(k)}) \setminus K(F)}$ as the lift of $\widehat \Lambda_{\kappa_0+j}|_{H_{k}(Y^{(\kappa_0+j})\setminus K(F)}$ (this part is $K$-qc) and $\widehat \Lambda_k|_{K(F)}=\psi$ (this part is conformal). It follows ~\cite[Lemma~10.4]{L2} and ~\cite[Lemma~2, page 303]{DH2} that $\widehat\Lambda_{k}$ is a $K$-qc map.
\end{proof}

\begin{Remark}\label{re:filled}In the proof of the above lemma, we know that the $K$-qc extension $\widehat \Lambda_k$ of $\Lambda_k$ can be made conformal a.e. on $K(F)$. Moreover, $\widehat \Lambda_k\circ f^{m_0} =\widetilde f^{m_0}\circ \widehat \Lambda_k$.
\end{Remark}

\begin{Lemma}\label{lem:pullback}Let $\Omega, \Omega_*,\widetilde \Omega, \widetilde \Omega_*$ be quasidisks.  
Let $H_*:\Omega_*\to \widetilde \Omega_*$ be a qc map and let $H:\partial \Omega\to \partial\widetilde \Omega$ be a homeomorphism. Assume $A$ and $\widetilde A$ are compact subset of $\Omega_*$ and $\widetilde \Omega_*$ respectively. If there exist  holomorphic proper maps $g:\Omega \to \Omega_*$ and $\widetilde g:\widetilde \Omega\to \widetilde \Omega_*$ of the same degree with the following properties:
\begin{itemize}
\item the critical values of $g$ and $\widetilde g$ lie in $A$ and $\widetilde A$ respectively;
\item $H\circ g=\widetilde g\circ\widetilde H$ on $\partial \Omega$.
\end{itemize}
Then $H$ has a $K$-qc extension from $\Omega$ onto $\widetilde \Omega$ where $K$ only depends on $H$, $A$ and $A'$. 
\end{Lemma}

\begin{proof}
Fix some $x \in \Omega_*$. Let $\phi_*:\Omega_*\to \mathbb D$ and $\widetilde \phi_*:\widetilde \Omega_*\to \mathbb D$ be conformal maps such that $\phi_*(x)=0$ and $\widetilde \phi_*(H_*(x))=0$. Choose $y\in g^{-1}(x)$ and $\widetilde y\in \widetilde g^{-1}(H_*(x))$. Let $\phi:\Omega\to \mathbb D$ and $\widetilde \phi:\widetilde \Omega\to \mathbb D$ be conformal maps such that $\phi(y)=0$ and $\widetilde \phi (\widetilde y)=0$. Let $\widehat H_*={\widetilde \phi_*}\circ H \circ{\phi_*}^{-1}:\mathbb D\to \mathbb D$ and $\widehat H={\widetilde \phi}\circ H \circ{\phi}^{-1}:\partial \mathbb D\to \partial \mathbb D$.  Since $A$ and $\widetilde A$ lie compactly in $\Omega_*$ and $\widetilde \Omega_*$ respectively, there exists $0<r=r(A,A',x,H_*(x))<1$ such that $\phi_*(A), \widetilde \phi_*(\widetilde A) \subset \mathbb D_{r^2}$. By \cite[Lemma~3.2]{AKLS}, $\widehat H$ has a $K$-qc extension $\widehat H:\mathbb D\to \mathbb D$. Thus $H$ has a $K$-qc extension from $\Omega$ onto $\widetilde \Omega$.
\end{proof}

\begin{Lemma}\label{lem:Julia}Let $f,\widetilde f$ and  $\{\Lambda_{k}\}$ be as in Lemma~\ref{lem:Fatou}. 
Then there exists $K>0$ and a sequence $\{k_n\}$ such that   $\Lambda_{k_n}|_{X^{(k_n)}}$ has a $K$-qc extension from $H_{k_n}(X^{(k_n)})$ onto $\widetilde H_{k_n}(X^{(k_n)})$.
\end{Lemma}
\begin{proof}
Recall that there exists $\kappa_0>0$ with the following properties.
 \begin{itemize}
 \item $\overline {Y^{(\kappa_0)}},\overline{f_0(Y^{(\kappa_0)})},\ldots, \overline{f^{m_0-1}_0(Y^{(\kappa_0)})}$ are disjoint;
 \item $f^{m_0}_0:Y^{(\kappa_0)} \to Y^{(\kappa_0-m_0)}$ is an APL map with filled Julia set $\overline{U_0}$.
  \end{itemize}
  We can further assume $\kappa_0$ is large enough so that $\tau_0 \notin  \bigcup_{j=1}^{m_0} f^j_0(\overline {Y^{(\kappa_0)}})$.
  
If  $c_0$ is combinatorially recurrent, i.e., for any $n\in \mathbb N$ there exists $\ell>0$ such that $f^{\ell}_0(c_0) \in X^{(n)}$, then the lemma follows from~\cite[Theorem~4.4]{AKLS}.

Now we assume $c_0$ is combinatorially non-recurrent. Then there exists $\ell_0>\kappa_0$ such that $f^j_0(c_0)\notin X^{(\ell_0)}$. 
For each $j\in \mathbb N$, let us denote $f^{j}_0(c_0)$ by $c_j$.

\medskip

{\bf Case 1.} $\partial U_0 \cap \omega(c_0) \ne \emptyset$. Then for any $n>\ell_0$, there exists a smallest non-negative integer $k'_n$ such that $c_{k'_n} \in Y^{(n)}$. Let $k_n=k'_n+\ell_0$. Consider the map $f^{k'_n}_0:Y^{(k_n)}(c_0) \to Y^{(\ell_0)}(c_{k'_n})=Y^{(\ell_0)}$. Note that this map is $d_1$-to-$1$. For otherwise, there exists a smallest positive integer $0<s<k'_n$ such that $f^{s}_0(Y^{(k_n)}(c_0))=Y^{(k_n-s)}(c_s)$ contains $0$. Thus $c_s\in Y^{(k_n-s)}$. By the definition of $k'_n$, $c_s\in Y^{(k_n-s)}\setminus Y^{(n)}$. Note that $F_0:=f^{m_0}_0:Y^{(\ell_0+m_0)} \to Y^{(\ell_0)}$ is a first return map. Since $f^{k'_n-s}_0(Y^{(k_n-s)})=Y^{(\ell_0)}$, there exists $u$ such that $F^u_0(Y^{(k_n-s)})=Y^{(\ell_0)}$. So $c_{k'_n}=f^{k'_n-s}_0(c_s)=F^{u}_0(c_s) \notin Y^{(n)}$ since  $F_0^{-1}(Y^{(n)}) \subset Y^{(n)}$  and $c_s\notin Y^{(n)}$.

Now take $n_0$ large so that $Y^{(n)} \Subset Y^{(\ell_0)}$ for all $n\ge n_0$. Since $f,\widetilde f\in \mathcal C(\lambda_{f_0})$, take a sufficiently large positive integer $k$, we know that 
\begin{itemize}
\item $H_{k}(Y^{(n)})\subset H_{n_0}(Y^{(n_0)}) \Subset H_{n_0}(Y^{(\ell_0)})$;
\item  $f^{k_n-\ell_0}:H_{k_n}(Y^{(k_n)}(c_0)) \to H_{k_n}(Y^{(\ell_0)}) $ is $d_1$-to-$1$ and its critical value lies in $H_k(Y^{(n)}) \subset  H_{n_0}(Y^{(n_0)}) \Subset H_{n_0}(Y^{(\ell_0)})$.
\end{itemize}
This statement also holds for the corresponding objects with tildes. Since $\widetilde f^{k'_n}\Lambda_{\ell_0}=\Lambda_{k'_n}\circ f^{k'_n}$ on $\partial H_{k_n}(Y^{(k_n)}(c_0))$, by Lemma~\ref{lem:pullback}, $H_{k_n}$ has a $K$-qc extension where $K$ only depends on $\Lambda_{\ell_0},H_{n_0}(Y^{(n_0)})$ and $\widetilde H_{n_0}(Y^{(n_0)})$.

\medskip

{\bf Case 2.} $\tau_0\in \omega(c_0)$. Let $q$ be the ray period of $\tau_0$. For each $n\in \mathbb N$, let $Y^{n}_*$ be the interior of  the closure of the union of all the puzzle pieces of depth $n$ which attaches $\tau_0$. Since $\tau_0$ is a repelling fixed point of $f^q_0$, there exist small disks $\tau_0 \in D\Subset D'\Subset Y^{\ell_0}_*$ such that $f^q_0:D\to D'$ is a conformal map. Let $g$ denote the inverse map of $f^q_0:D\to D'$. Since $Y^n_*$ shrinks to $\tau_0$, there exists $n_1>\ell_0$ large such that $Y^{n_1}_* \subset D'$. Then there exists $s_1$ and $n_2$ such that $Y^{n_2}_*=g^{s_1}(Y^{n_1}_*) \Subset Y^{n_1}_*$ as $g$ is contracting. Since $\tau_0\in \omega(c_0)$, for any $n$, there exists $t>n+n_2$ such that $c_{t} \in Y^{n_2}_*$. Then there exists $t'>t$ so that $c_{t'} \in Y^{n_1}_* \setminus Y^{n_2}_*$. Note that $\tau_0$ lies on the boundary of $Y^{(\ell_0)}(c_{t'})$ since $n_1>\ell_0$. Let $k_n=t'+\ell_0$. We consider the map $f^{t'}_0:Y^{(k_n)}(c_0) \to Y^{(\ell_0)}(c_{t'})$. We claim this map is $(d_2+1)$-to-$1$. For otherwise, there exists $0<s'<t'$ such that $c_{s'} \in Y^{(t'+\ell_0-s')} \subset Y^{(\kappa_0)}$. Since $\tau_0 \notin  \bigcup_{j=1}^{m_0} f^j_0(\overline {Y^{(\kappa_0)}})$, we conclude that  $\tau_0$ does not lie on the boundary of $Y^{(\ell_0)}(c_{t'})=f^{t'-s'}_0(Y^{(t'+\ell_0-s')})$. This is a contradiction.
Again take a sufficiently large positive integer $k$, we know that 
\begin{itemize}
\item $H_{k}(Y^{n_1}_* \setminus Y^{n_2}_*) =H_{n_2}(Y^{n_1}_* \setminus Y^{n_2}_*) \Subset H_{n_2}(Y^{(\ell_0)}_*)$;
\item  $f^{k_n-\ell_0}:H_{k_n}(Y^{(k_n)}(c_0)) \to H_{k_n}(Y^{(\ell_0)}(c_{k_n-\ell_0})) $ is $(d_2+1)$-to-$1$ and its critical value lies in $H_{n_2}(Y^{n_1}_* \setminus Y^{n_2}_*) $.
\end{itemize}
We claim that the critical value $v_n$ of $f^{k_n-\ell_0}|_{H_{k_n}(Y^{(k_n)}(c_0))}$ lies in a definite compact subset of $H_{\ell_0}(Y^{(\ell_0)}(c_{k_n-\ell_0}))$. For otherwise, $
\{v_n\}$ accumulates on  the boundary of $H_{\ell_0}(Y^{(\ell_0)}(c_{k_{n}-\ell_0}))$. Since $v_{n}\in H_{n_2}(Y^{n_1}_* \setminus Y^{n_2}_*)$, the only possibility is that $\{v_n\}$  accumulates on the external rays truncated by height ${(d_1+d_2)}^{-n_1}$ and $(d_1+d_2)^{-n_2}$. This implies $v_n$ lies in the Fatou set of $f$ for infinitely many $n$, which is a contradiction. The above  statements also hold for the corresponding objects with tildes. Since $\widetilde f^{k_n-\ell_0}\Lambda_{\ell_0}=\Lambda_{k_n-\ell_0}\circ f^{k_n-\ell_0}$ on $\partial H_{k_n}(Y^{(k_n)}(c_0))$, by Lemma~\ref{lem:pullback}, $H_{k_n}$ has a $K$-qc extension.

\medskip

{\bf Case 3.} $\tau_0\notin \omega(c_0)$ and $\partial U_0 \cap \omega(c_0) = \emptyset$. In this case we can choose $\ell'>\ell_0$ large enough so that $c_j\notin Y^{(\ell')}$ for all $j\in \mathbb N$.  For any $n\in \mathbb N$, let $k_n=n+\ell'$. Let us consider $f^{k_n-\ell'}_0:Y^{(k_n)}(c_0) \to Y^{\ell'}(c_{n})$. Clearly, it is unicritical and $(d_2+1)$-to-$1$. Since $\tau_0\notin \omega(c_0)$, the critical vaule $c_n$ lies in a definite compact subset of $Y^{\ell'}(c_{n})$ (by a similar argument in Case~2).  Since $f \in \mathcal C(\lambda_{f_0})$, we know  $f^{k_n-\ell'}:H_{k_n}(Y^{(k_n)}(c_0)) \to H_{\ell'}(Y^{\ell'}(c_{n}))$ is also  $(d_2+1)$-to-$1$ and its critical value lies in a definite compact subset $ H_{\ell'}(Y^{\ell'}(c_{n}))$ (by a similar argument in Case~2). This statement also holds for the corresponding objects with tildes. Since $\widetilde f^{k_n-\ell'}\Lambda_{\ell'}=\Lambda_{k_n-\ell'}\circ f^{k_n-\ell'}$ on $\partial H_{k_n}(Y^{(k_n)}(c_0))$, by Lemma~\ref{lem:pullback}, $H_{k_n}$ has a $K$-qc extension.
\end{proof}

\begin{Lemma}\label{lem:pseudo}Let $f, \widetilde f\in \mathcal C(\lambda_{f_0})$.  Assume $H$ $(\text{resp.}~\widetilde H)$ is weak pseudo-conjugacy up to depth $n$ between $f_0$ and $f$ $(\text{resp.}~\widetilde f)$ for some $n>0$. Let $\Lambda=\widetilde H\circ H^{-1}$. If $\Lambda|_{\partial H(Y^{(n)}\cup \partial X^{(n)})}$ admits a $K$-qc extension to $\widetilde H(Y^{(n)} \cup X^{(n)})$, then there exists a $K$-qc map $\widehat\Lambda$ such that $\widehat \Lambda$ is identity in the  B\"ottcher coordinate for $f$ and $\widetilde f$ near $\infty$ and $\widehat \Lambda\circ f=\widetilde f\circ \widehat \Lambda$ on $\C\setminus H(Y^{(n)}\cup \partial X^{(n)})$.

\end{Lemma}
\begin{proof}
Without loss of generality, we assume $\Lambda$ is $K$-qc on $H(Y^{(n)}\cup X^{(n)})$.
Let $L_n$ denote the domain of the first landing map to $H(Y^{(n)}\cup X^{(n)})$:
\[L_n=\{z\mid \exists k\ge 0~\text{such that}~f^{k}(z) \in H(Y^{(n)}\cup X^{(n)})\}.\] 
For each component $W$ of $L_n$, we define $\phi_W$ be identity  in the B\"ottcher coordinate for $f$ and $\widetilde f$ on $\partial W$ and interpolate homemorphicly. Note that $\phi_W(W)$ is a first landing domain of $H(Y^{(n)}\cup X^{(n)})$ since $\lambda_{f_0}\subset \lambda_{f}\cap \lambda_{\widetilde f} $. 
Thus we can define $\widehat \Lambda|_W:W\to \phi_W(W)$ as the univalent pullback of $H$. Then define $\widehat \Lambda$ on $\C\setminus (K(f)\cup L_n)$ as identity   in the B\"ottcher coordinate for $f$ and $\widetilde f$.  These two maps match on the common boundary since $\Lambda|_{\partial H(Y^{(n)} \cup X^{(n)})}$ is identity in the B\"ottcher coordinate for $f$ and $\widetilde f$. Note that $\widehat \Lambda$ is $K$-qc. By \cite[Fact~5.1]{KSS}, the residual set $$E(H(Y^{(n)}\cup X^{(n)}):=\{z\in K(f)\mid f^j(z) \notin H(Y^{(n)} \cup X^{(n)})~\text {for all}~j\in\mathbb N \}$$
is a nowhere dense compact set with Lebesgue measure $0$. Thus $\widehat \Lambda$ admits a $K$-qc extension to the whole complex plane $\C$.
\end{proof}

Let $f\in \mathcal C(\lambda_{f_0})$ For any $n\in\mathbb N$, let $H_n$ be a weak pseudo-conjugacy  between $f_0$ and $f$ up to depth $n$. Define $\mathcal K(f):=\bigcap_{n\in \mathbb N} H_n(Y^{(n)})$ and  $\mathcal S(f):=\bigcap_{n\in \mathbb N} H_n(X^{(n)})$. By using the same argument as in the proof of Lemma~\ref{lem:X shrinking}, one can show $\mathcal S(f)$ is a singleton which is a Julia critical point $c'_f$ of $f$.

\begin{Lemma}\label{lem:measure} The set $E(f):=\{z\in \C\setminus \mathcal K(f)\mid \omega(z) \cap (\mathcal K(f)\cup \mathcal S(f)) \ne \emptyset\}$ has zero Lebesgue measure for all $f\in \mathcal C(\lambda_{f_0})$.
\end{Lemma}
\begin{proof} 

{\bf Case 1.} $c_0$ is combinatorially recurrent. 

We first show that for $E'(f)=\{z\in \C\mid \omega(z) \cap \mathcal S(f) \ne \emptyset\}$ has zero Lebesgue measure. To this end, we construct  a generalized polynomial like map (complex box map) as in the proof of Lemma~\ref{lem:X shrinking}. Let $P_0=X^{(\kappa_0)}$. For $n\ge 1$, we define $P_n$ as the component of the first return map to $P_{n-1}$ which contains $c_0$ inductively.  It is well known this is a moment $n_0$ so that $P_{n_0+1} \Subset P_{n_0}$ since $f_0$ is not renormalizable at $c_0$ (see~\cite[Lemma~2.2]{KS} for an example). Let $H_{n_0+1}$ be a weak pseudo-conjugacy between $f_0$ and $f$ up to depth $n_0+1$. Let $V_f:=H_{n_0+1}(P_{n_0+1})$  and let $U_f$ be the domain of the first return to $V_f$ under $f$. The first return map $R_f:U_f\to V_f$ is our desired generalized polynomial-like map. Similarly, let $V_0:=P_{n_0+1}$ and $R_0:U_0\to V_0$ be the first return map under $f_0$. Since $f_0$ is non-renormalizable at $c_0$, so is $R_0:U_0\to V_0$. As $f\in \mathcal C(\lambda_{f_0})$, $R_f:U_f\to V_f$ is also non-renormalizable. By~\cite[Theorem~1.5]{KS} (see also \cite{L3}), $K(R_f)$ has Lebesgue measure $0$ where $K(R_f):=\{z\in U_f \mid R^n_f(z) \in U_f ~\text{for all}~n\in \mathbb N\}$. Note that $E'(f)\subset K(R_f)$, so $E'(f)$ has zero Lebesgue measure.

Now we show $E(f)\setminus E'(f)$ has zero Lebesgue measure. It suffices to show that for any $z\in E(f)\setminus E'(f)$, $z$ is not a density point of 
 $E(f)\setminus E'(f)$.  Let $z_k:=f^k(z)$ for all $k\in \mathbb N$. Since $z\notin \mathcal S(f)$, there exists $\ell>0$ such that $\mathrm{orb}(z) \cap H_{\ell}(X^{(\ell)}))=\emptyset$. Now consider the APL map $f^{m_0}:H_{\ell+m_0}(Y^{(\ell+m_0)}) \to H_{\ell}(Y^{(\ell)})$. 
Choose $n_1$ large so that  $H_{n_1}(Y^{(n_1)}) \Subset H_{\ell}(Y^{(\ell)})$.  Let $\delta$ denote the  Hausdorff distance between $\overline H_{n_1}(Y^{(n_1)})\setminus H_{n_1+m_0}(Y^{(n_1+m_0)})$ and $\partial H_{\ell}(Y^{(\ell)}) \cup \mathcal K(f)$.  As they are disjoint closed set, $\delta>0$.
Since $\omega(z) \cap \mathcal K(f) \ne \emptyset$, there is a strictly increasing  sequence $\{k_n\}_{n\in \mathbb N}$ such that $z_{k_n} \in  H_{n_1}(Y^{(n_1)})\setminus \overline{H_{n_1+m_0}(Y^{(n_1+m_0)})}$. Let $D_n$ be the component of $f^{-k_n}(\mathbb D(z_{k_n},\delta))$ which contains $z$ for all $n\in \mathbb N$. Note that $f^{k_n}:D_n \to \mathbb D(z_{k_n},\delta)$ is conformal. Indeed, any pullback of $\mathbb D(z_{k_n},\delta)$ cannot intersect $H_{\ell}(X^{(\ell)}))$  since  $\mathrm{orb}(z) \cap H_{\ell}(X^{(\ell)}))=\emptyset$ and $\mathbb D(z_{k_n},\delta) \subset H_{\ell}(Y^{(\ell)})$. By the Koebe Distortion Theorem,  $z$ is not a density point of $K(f)$.

{\bf Case 2.} $c_0$ is combinatorially non-recurrent. Fix $\ell>\kappa_0$ large so that $f^j_0(c_0) \notin X^{(\ell)}$ for all $j\ge 1$. 

In this case, we first show $E''(f)=\{z\in E(f)\mid \omega(z) \cap \mathcal K(f) \ne \emptyset\}$ has zero Lebesgue measure. The proof is similar to the second part of Case~1. Take $z\in E''(f)$ and let $z_k:=f^k(z)$ for all $k\in \mathbb N$. Choose $n_1$ large so that  $H_{n_1}(Y^{(n_1)}) \Subset H_{\ell}(Y^{(\ell)})$.  Let $\delta$ denote the  Hausdorff distance between $\overline H_{n_1}(Y^{(n_1)})\setminus H_{n_1+m_0}(Y^{(n_1+m_0)})$ and $\partial H_{\ell}(Y^{(\ell)}) \cup \mathcal K(f)$.  As they are disjoint closed set, $\delta>0$. Since $\omega(z) \cap \mathcal K(f) \ne \emptyset$, there is a strictly increasing  sequence $\{k_n\}_{n\in \mathbb N}$ such that $z_{k_n} \in  H_{n_1}(Y^{(n_1)})\setminus \overline{H_{n_1+m_0}(Y^{(n_1+m_0)})}$. Let $D_n$ and $D'_n$ be the component of $f^{-k_n}(\mathbb D(z_{k_n},\delta))$ and $f^{-k_n}(\mathbb D(z_{k_n},\delta/2))$ which contains $z$ for all $n\in \mathbb N$ respectively. Note that $f^{k_n}:D_n \to \mathbb D(z_{k_n},\delta)$ is a holomorphic proper map of degree at most $d_2+1$ (since $c_0$ is combinatorially non-recurrent). Thus $f^{k_n}:D'_n \to \mathbb D(z_{k_n},\delta/2)$ can be written as $\psi_n\circ Q_0\circ \phi_n$ where $Q_0(z)=z^{d_2+1}$, $\psi_n$ and $\phi_n$ are conformal maps with uniformly bounded distortion. Then it follows $z$ is not a density point of $E''(f)$ by a standard argument.

It remains to show that $E(f)\setminus E''(f)$ has zero Lebesgue measure. Take $z\in E(f)\setminus E''(f)$ and let $z_k:=f^k(z)$ for all $k\in \mathbb N$. Fix $n_2$ large so that $H_{n_2}(X^{(n_2)}) \Subset H_{\ell}(X^{(\ell)})$. Let $\delta_1$ denote the Hausdorff distance between $\omega(z)$ and $\mathcal K(f)$ and let $\delta_2$ denote the Hausdorff distance between $\partial H_{n_2}(X^{(n_2)})$ and $\partial H_{\ell}(X^{(\ell)})$. Take $\delta<\min (\delta_1,\delta)$. Since $\omega(z) \cap \mathcal S(f)\ne \emptyset$, there exists a strictly increasing sequence $\{k_n\}$ such that $z_{k_n}\in H_{n_2}(X^{(n_2)})$. Let $U_n$ be the component of $f^{-k_n}(\mathbb D(z_{k_n},\delta))$ which contains $z$ for all $n\in \mathbb N$. Clearly, $f^{k_n}:U_n\to \mathbb D(z_{k_n},\delta)$ is conformal by the construction. It follows from Koebe Distortion Theorem that $z$ is not a density point of $E(f)\setminus E''(f)$.
\end{proof}

We are now in a position to show Theorem~\ref{thm:injection}.
\begin{proof}[Proof of Theorem~\ref{thm:injection}]
Assume there are $f\in \mathcal C(\lambda_{f_0})$ and $\widetilde f \in \mathcal C(\lambda_{f_0})$  such that $\chi(f)=\chi(\widetilde f)$. We prove that $f=\widetilde f$. 

For any $n\in \mathbb N$, let  $H_n$ be a qc weak pesudo-conjugacy between $f_0$ and $f$ and   let  $\widetilde H_n$ be a qc weak pesudo-conjugacy between $f_0$ and $\widetilde f$. Let $\Lambda_n=\widetilde H_n\circ H^{-1}_n$ for all $n\in\mathbb N$. It follow there exists a sequence $\{k_n\}$ and $K>0$ with the following properties:
\begin{itemize}
\item $\Lambda_{k_n}$ has a $K$-qc extension from $H_{k_n}(Y^{(k_n)} \cup X^{(k_n)})$ onto $\widetilde H_{k_n}(Y^{(k_n)} \cup X^{(k_n)})$;
\item $\bar\partial \Lambda_{k_n}=0$ a.e. on $\mathcal K(f)$
\end{itemize}
by Lemma~\ref{lem:Fatou}, Lemma~\ref{lem:Julia} and Remark~\ref{re:filled}. Then by Lemma~\ref{lem:pseudo}, we can further assume $\Lambda_{k_n}$ is global $K$-qc and $\Lambda_{k_n}\circ f=\widetilde f\circ  \Lambda_{k_n}$ on $\C\setminus H(Y^{(n)}\cup \partial X^{(n)})$. By the compactness of normalized $K$-qc maps, $\{\Lambda_{k_n}\}$ has a convergent subsequence. Without loss of generality, we assume $\Lambda_{k_n} $ converges to a $K$-qc map $\Phi:\C\to \C$. Clearly $\Phi\circ f=\widetilde f\circ \Phi$ on $\C\setminus (\mathcal K(f)\cup \mathcal S(f))$.
By the construction in the proof of Lemma~\ref{lem:pseudo}, we know that $\bar\partial \Lambda_{k_n}$ only supports on $\bigcup_{j=0}^{\infty}f^{-j}(H_{k_n}(Y^{(k_n)}\cup X^{(k_n)}\setminus \mathcal K(f)))$. Since these sets shrink to $E(f)$ which has zero Lebesgue measure by Lemma~\ref{lem:measure}, we conclude $\Phi$ is conformal. By the proof of Lemma~\ref{lem:Fatou} and Remark~\ref{re:filled}, we know $\Phi \circ f^{m_0}=\widetilde f^{m_0}\circ \Phi$ on $\mathcal K(f)$. Since $\Phi\circ f=\widetilde f\circ \Phi$ on $\bigcup_{j=1}^{m_0-1} f^j(\mathcal K(f))$, we have 
\[\widetilde f^{m_0-1}\circ \widetilde f_0\circ \Phi= \Phi\circ f^{m_0}=\Phi\circ f^{m_0-1}\circ f=\widetilde f^{m_0-1}\circ \Phi\circ f\]
on $\mathcal K(f)$. As $\widetilde f^{m_0-1}|_{\widetilde f(\mathcal K(\widetilde f))}$ is one-to-one, we conclue $\Phi\circ f=\widetilde f\circ \Phi$ on $\mathcal K(f)$. Thus,  $\Phi\circ f=\widetilde f\circ \Phi$ on the whole complex plane $\C$. Since $\Phi$ is conformal and is tangent to $id$ at $\infty$, $\Phi$ is identity.
\end{proof}

Theorem~A follows immediately from Theorem~\ref{thm:surjection} and Theorem~\ref{thm:injection}.  Theorem~C is a consequence of Theorem~A, \cite[Propsotion~4.7]{McM2} and the compactness of $\mathcal C(\lambda_{f_0})$ (\cite[Theorem~8.1]{IK}).

\bibliographystyle{plain}             

\end{document}